\documentclass{article}
\title{\textbf{The nucleus of the Grassmann graph $J_q(N,D)$}}
\author{Jae-Ho Lee, Jongyook Park, Ian Seong\footnote{All authors contributed equally to this paper.}}
\date{}
\usepackage {parskip}
\usepackage [margin=2.5cm] {geometry}
\usepackage{enumitem}
\usepackage[utf8]{inputenc}
\usepackage[T1]{fontenc}
\usepackage{geometry}
\usepackage{amsmath, amssymb, amsthm}
\usepackage{eucal}
\usepackage{hyperref}
\usepackage{graphicx}
\usepackage{mathtools}
\usepackage{relsize}
\usepackage{bm}
\usepackage{esvect}
\usepackage{circuitikz}
\usepackage{tikz}
\usepackage{hhline}
\usepackage{enumitem}
\usepackage{comment}
\usepackage{ytableau}
\usepackage{pdflscape}
\usepackage{mathabx}
\usepackage{multicol}

\def\mult{\text{\rm{mult}}}
\newtheorem{theorem}{Theorem}[section]
\newtheorem{corollary}[theorem]{Corollary}
\newtheorem{lemma}[theorem]{Lemma}
\newtheorem{prop}[theorem]{Proposition}
\theoremstyle{definition}
\newtheorem{definition}[theorem]{Definition}

\newtheorem{problem}[theorem]{Problem}
\newtheorem{remark}[theorem]{Remark}

\begin{document}
\maketitle
\begin{abstract}
    Let $\mathbb{F}_q$ denote a finite field with $q$ elements. 
Let $N$ and $D$ denote integers with $N>D \ge 1$. 
Let $\mathcal{V}$ denote an $N$-dimensional vector space over $\mathbb{F}_q$. 
The Grassmann graph $J_q(N,D)$ is the graph with vertex set $X$ that consists of the $D$-dimensional subspaces of $\mathcal{V}$. 
Two vertices are adjacent whenever their intersection has dimension $D-1$. 
Fix a vertex $x$ in $X$. 
The Terwilliger algebra $T=T(x)$ of $J_q(N,D)$ with respect to $x$ is the subalgebra of $\mathrm{Mat}_X(\mathbb{C})$ generated by the adjacency matrix $A$ and the dual adjacency matrix $A^* = A^*(x)$. 
It is known that an irreducible $T$-module $W$ has certain parameters called the endpoint $r$, the dual endpoint $t$, and the diameter $d$. 
The displacement of $W$ is defined to be the integer $r+t-D+d$. 
Let $\mathcal{N}=\mathcal{N}(x)$ denote the span of all irreducible $T$-modules with displacement 0. 
We call $\mathcal{N}$ the nucleus of $J_q(N,D)$ with respect to $x$.
In this paper, we study the structure of $\mathcal{N}$. 
Specifically, we present a formula for the dimension of $\mathcal{N}$, construct two explicit bases for $\mathcal{N}$, and describe the action of $A$ and $A^*$ on these bases. 
To obtain these results, we use the projective geometry $P_q(N)$, consisting of all subspaces of $\mathcal{V}$, as a key tool. \\ \\
     \textbf{Keywords.} $Q$-polynomial distance-regular graph; projective geometry; Grassmann graph; Terwilliger algebra.\\
    \textbf{2020 Mathematics Subject Classification.} Primary: 05E30. Secondary: 05C50, 06C05.
\end{abstract}

\section{Introduction}
This paper is about a family of distance-regular graphs \cite{BBIT, BCN} called the Grassmann graphs \cite{GK2025JCT, Metsch, Seong2024GC, Seong2025EJC}. These graphs are known to have the $Q$-polynomial property \cite{Cerzo, TKCP2022GC, Ter2024arxiv}. In \cite{Ter1992JOACI, Ter1993JOACIII}, Terwilliger introduces the Terwilliger algebra $T=T(x)$ of a distance-regular graph with respect to a given vertex $x$, and described the irreducible $T$-modules in the standard module. In \cite{Ter2005GC}, he defined the notion of the displacement of an irreducible $T$-module. In \cite{Ter2024arxiv}, Terwilliger introduces a certain $T$-module of a $Q$-polynomial distance-regular graph, called the nucleus. In particular, he describes the nucleus of a dual polar graph. In \cite{NT2024arxiv}, Nomura and Terwilliger describe the nucleus of a Johnson graph. 

The goal of this paper is to describe the nucleus of a Grassmann graph. To set the stage, we first recall the notion of displacement. Let $\Gamma$ denote a $Q$-polynomial distance-regular graph with vertex set $X$, path-length distance function $\partial$, and diameter $D$. Fix a vertex $x \in X$. The Terwilliger algebra $T=T(x)$ of $\Gamma$ with respect to $x$ is the subalgebra of $\operatorname{Mat}_{X}(\mathbb{C})$ generated by the adjacency matrix $A$ and the dual adjacency matrix $A^*=A^*(x)$ \cite{Ter1992JOACI}. An irreducible $T$-module $W$ in a standard $\operatorname{Mat}_{X}(\mathbb{C})$-module has certain parameters called the endpoint, dual endpoint, and diameter, denoted by $r, t, d$ respectively; see \cite[Definition~3.5,~Lemmas~3.9,~3.12]{Ter1992JOACI}. The displacement of $W$ is defined to be the integer $r+t-D+d$; this integer is known to be nonnegative \cite[Proposition~6.5]{Ter2024arxiv}. We now define the nucleus of $\Gamma$. The nucleus $\mathcal{N}=\mathcal{N}(x)$ with respect to $x$ is the span of all irreducible $T$-modules that have displacement $0$. 

The Grassmann graph is defined from a projective geometry, which we now describe. 
Let $\mathbb{F}_q$ denote a finite field with $q$ elements. 
Let $N\geq 1$, and let $\mathcal{V}$ denote a vector space of dimension $N$ over $\mathbb{F}_q$. 
Let $P = P_q(N)$ denote the set of all subspaces of $\mathcal{V}$. 
The set $P$, together with inclusion partial order, forms a poset called a projective geometry. 
Now assume $N > D \geq 1$.
The Grassmann graph $J_q(N,D)$ has the vertex set $X$, consisting of the $D$-dimensional subspaces of $\mathcal{V}$. 
Two vertices $y, z \in X$ are adjacent whenever the subspace $y \cap z$ has dimension $D-1$.

To describe the nucleus of $J_q(N,D)$, we make heavy use of the results in \cite{LIW2020LAA}, which we now summarize. 
For the rest of this section, assume that $N>2D$. 
Recall the projective geometry $P$. 
Pick $x\in P$ such that $\dim x=D$. 
Note that $x$ is a vertex of $J_q(N,D)$.
For $0 \leq i\leq D$ and $0 \leq j\leq N-D$ define 
    \begin{equation*}
        P_{i,j}= \bigl\{u\in P\mid \dim (u\cap x)=i, \dim u=i+j\bigr\}.
    \end{equation*}
    
Let $u,v\in P$ such that $u\subseteq v$ and $\dim v=\dim u+1$. Write
    \begin{equation*}
        u\in P_{i,j}, \qquad \qquad v\in P_{\ell,n}.
    \end{equation*}
    By \cite[Lemma~2.3]{Seong2024arxiv}, either (i) $\ell=i+1$ and $n=j$, or (ii) $\ell=i$ and $n=j+1$. We say that $v$ $\slash$-covers $u$ whenever (i) holds, and $v$ $\backslash$-covers $u$ whenever (ii) holds. 

    For $0 \leq i\leq D$ and $0 \leq j\leq N-D$, define the diagonal matrix $E^{*}_{i,j}\in \text{Mat}_P(\mathbb{C})$ as follows. For $u\in P$, the $(u,u)$-entry of $E^{*}_{i,j}$ is
    \begin{equation*}
        \bigl(E^{*}_{i,j}\bigr)_{u,u}=\begin{cases}
            1&\text{if }u\in P_{i,j},\\
            0&\text{if $u\notin P_{i,j}$}.
        \end{cases}
    \end{equation*}  

    Define diagonal matrices $K_1,K_2\in \text{Mat}_P(\mathbb{C})$ as follows. For $u\in P$ their $(u,u)$-entries are
    \begin{equation*}
        (K_1)_{u,u}=q^{\frac{D}{2}-i},\qquad \qquad (K_2)_{u,u}=q^{\frac{N-D}{2}-j},
    \end{equation*}
    where $u\in P_{i,j}$. The matrices $K_1, K_2$ are invertible. Define matrices $L_1,L_2,R_1,R_2\in \text{Mat}_P(\mathbb{C})$ as follows. For $u,v\in P$ their $(u,v)$-entries are
    \begin{align*}
        \bigl(L_1\bigr)_{u,v}&=\begin{cases}
            1&\text{if }v\text{ $\slash$-covers }u,\\
            0&\text{if }v\text{ does not $\slash$-cover }u,
        \end{cases}\\
        \bigl(L_2\bigr)_{u,v}&=\begin{cases}
            1&\text{if }v\text{ $\backslash$-covers }u,\\
            0&\text{if }v\text{ does not $\backslash$-cover }u,
        \end{cases}\\
        \bigl(R_1\bigr)_{u,v}&=\begin{cases}
            1&\text{if }u\text{ $\slash$-covers }v,\\
            0&\text{if }u\text{ does not $\slash$-cover }v,
        \end{cases}\\
        \bigl(R_2\bigr)_{u,v}&=\begin{cases}
            1&\text{if }u\text{ $\backslash$-covers }v,\\
            0&\text{if }u\text{ does not $\backslash$-cover }v.
        \end{cases}
    \end{align*}
    Let $\mathcal{H}=\mathcal{H}(x)$ denote the subalgebra of $\text{Mat}_{P}(\mathbb{C})$ generated by $L_1,L_2,R_1,R_2, K_1^{\pm1}, K_2^{\pm 1}$. Let $\Psi$ denote the standard $\text{Mat}_{P}(\mathbb{C})$-module. Note that $\Psi$ is an $\mathcal{H}$-module. We also note that the vector space $\Psi$ is a direct sum of irreducible $\mathcal{H}$-modules. Let $\Omega$ denote an irreducible $\mathcal{H}$-module. The isomorphism class of $\Omega$ is determined by the parameters $\alpha,\beta,\rho$ (see \cite[p.~133]{LIW2020LAA}); we call the triple $(\alpha,\beta,\rho)$ the type of $\Omega$. Consider the subspace 
    \begin{equation*}
    \sum_{j=0}^{D}E^*_{D-j,j}\Omega
    \end{equation*}
    restricted to the standard $\text{Mat}_{X}(\mathbb{C})$-module $V$. By \cite[Proposition~5.5]{LIW2020LAA}, this subspace is an irreducible $T$-module of $J_q(N,D)$ with respect to $x$.
    
We now summarize the main results of this paper. We continue to fix $x\in P$ such that $\dim x=D$. We say that an irreducible $\mathcal{H}$-module $\Omega$ is alpha-dominant whenever $\Omega$ has type $(\alpha,0,0)$ for $0\leq \alpha\leq \frac{D}{2}$. Recall the nucleus $\mathcal{N}$ of $J_q(N,D)$. We show that there is a one-to-one correspondence between the following two sets:
\begin{enumerate}[label=(\roman*)]
    \item irreducible $T$-submodules of $\mathcal{N}$;
    \item irreducible $\mathcal{H}$-modules that are alpha-dominant.
\end{enumerate}
Using this fact, we show that the dimension of $\mathcal{N}$ is equal to
\begin{equation*}
    \sum_{i=0}^{D}\binom{D}{i}_q,
\end{equation*}
where $\binom{}{}_q$ is the $q$-binomial coefficient. For $y \in X$, let $\widehat{y}$ denote the vector in $V$ with $1$ in the $y$-coordinate and $0$ elsewhere.
For $Y \subseteq X$, define 
\begin{equation*}
    \widehat{Y} = \sum_{y \in Y} \widehat{y}.
\end{equation*}
This is called the characteristic vector of $Y$. Fix $0\leq i\leq D$. Pick $\alpha\subseteq x$ such that $\dim \alpha=D-i$. Define the vectors
\begin{equation*}
    \alpha^{\vee}=\sum_{\alpha\subseteq y}\widehat{y}, \qquad \qquad \alpha^{\mathcal{N}}=\sum_{x\cap y=\alpha}\widehat{y}.
\end{equation*}

Define
\begin{equation*}
    \Gamma_i(x)=\{y\in X\mid \partial(x,y)=i\}.
\end{equation*}

Define a subgraph $\gamma_i(x)$ of $\Gamma_i(x)$ as follows: remove from $\Gamma_i(x)$ the edges $yz$ for which $y\cap z$ is $\slash$-covered by each of $y,z$. We show that $\gamma_i(x)$ consists of $\binom{D}{i}_q$ connected components; each component corresponds to a subspace $\alpha\subseteq x$ of dimension $D-i$, in the sense that the characteristic vector of the component is $\alpha^{\mathcal{N}}$. We then present the action of the adjacency matrix $A$ and the dual adjacency matrix $A^*$ on the vectors $\alpha^{\vee}, \alpha^{\mathcal{N}}$.
We show that each of the following two sets forms a basis for $\mathcal{N}$:
\begin{enumerate}[label=(\roman*)]
    \item $\{\alpha^{\vee}\mid \alpha\subseteq x\}$;
        \item $\bigl\{\alpha^{\mathcal{N}}\mid \alpha\subseteq x\bigr\}$.
\end{enumerate}
We finish the paper with an appedix that contains identities involving $q$-binomial coefficients.

This paper is organized as follows. In Section 2 we present some preliminaries on a $Q$-polynomial distance-regular graph $\Gamma$ and its Terwilliger algebra $T$. In Section 3 we discuss the nucleus of $\Gamma$. In Section 4 we discuss the projective geometry $P$ and the Grassmann graph $J_q(N,D)$. In Section 5 we discuss the algebra $\mathcal{H}$ and the irreducible $\mathcal{H}$-modules. In Section 6 we discuss the irreducible $T$-modules of $J_q(N,D)$ and its nucleus. In Section 7 we discuss the graph $\gamma_i(x)$ and the vectors $\alpha^{\vee}, \alpha^{\mathcal{N}}$. In Section 8 we discuss the action of $A,A^*$ on the vectors $\alpha^{\vee}, \alpha^{\mathcal{N}}$. In Section 9 we find the dimension of $\mathcal{N}$ and display two of its bases. In Section 10 we briefly discuss the case $N=2D$. Section 11 is the appendix on the identities that involve $q$-binomial coefficients.

\section{Preliminaries}
We now begin our formal argument. Let $\Gamma=(X,R)$ denote a connected graph with no loops or multiple edges, vertex set $X$, edge set $R$, and path-length distance function $\partial$. By the \emph{diameter} $D$ of $\Gamma$, we mean the value $\max\{\partial(x,y)\mid x,y\in X\}$. To avoid triviality, we assume that $D\geq 1$ for the rest of the paper. 

For $0\leq i\leq D$ and a vertex $x\in X$, let 
\begin{equation*}
    \Gamma_{i}(x)=\{y\in X\mid \partial(x,y)=i\}.
\end{equation*}
We call $\Gamma_i(x)$ the \emph{$i$-th subconstituent of $\Gamma$ with respect to $x$}. We say that $\Gamma$ is \emph{regular with valency $k$} whenever $\vert \Gamma_1(x)\vert=k$ for all $x\in X$. We say that $\Gamma$ is \emph{distance-regular} whenever for all $0\leq h,i,j\leq D$ and all $x,y\in X$ such that $\partial(x,y)=h$, the cardinality of the set $$\{z\in X\mid\partial(x,z)=i,\; \partial(y,z)=j\}$$ depends only on $h,i,j$. This cardinality is denoted by $p_{i,j}^{h}$. From now on, we assume that $\Gamma$ is distance-regular. Observe that $\Gamma$ is regular with valency $k=p^{0}_{1,1}$.

Define 
\begin{equation*}
        b_i=p^{i}_{1,i+1}\; \;  (0\leq i<D),\qquad \qquad a_i=p^{i}_{1,i} \; \; (0\leq i\leq D),\qquad \qquad c_i=p^{i}_{1,i-1}\; \; (0<i\leq D).
\end{equation*}
Note that $b_0=k$, $a_0=0$, $c_1=1$. Also note that  
\begin{equation*}
    b_i+a_i+c_i=k \qquad \qquad (0\leq i\leq D),
\end{equation*}
where $c_0=0$ and $b_D=0$. We call $b_i$, $a_i$, $c_i$ the \emph{intersection numbers of $\Gamma$}.

Let $\text{Mat}_{X}(\mathbb{C})$ denote the $\mathbb{C}$-algebra consisting of the matrices with rows and columns indexed by $X$ and all entries in $\mathbb{C}$. 
For $y,z \in X$, the $(y,z)$-entry of $A_i$ is
\begin{equation*}
    (A_i)_{y,z}=\begin{cases}
        1&\text{if }\partial(y,z)=i,\\
        0&\text{if }\partial(y,z)\neq i.
    \end{cases}
\end{equation*}
We call $A=A_1$ the \emph{adjacency matrix of $\Gamma$}. By \cite[Section~3]{NT2024arxiv}, the following (i)--(v) hold:
\begin{enumerate}[label=(\roman*)]
    \item $A_0=I$, where $I$ is the identity matrix in $\text{Mat}_{X}(\mathbb{C})$;
    \item $J=\sum_{i=0}^{D} A_i$, where $J$ is the all-$1$ matrix;
    \item $A_i^{\top}=A_i$ ($0\leq i\leq D$), where $\top$ denotes matrix transpose;
    \item $\overline{A_i}=A_i$($0\leq i\leq D$), where $-$ denotes complex conjugate;
    \item $A_iA_j=A_jA_i=\sum_{h=0}^{D}p^{h}_{i,j}A_h$ ($0\leq i,j\leq D$).
\end{enumerate}
Hence, the matrices $\{A_i\}_{i=0}^{D}$ form a basis for a commutative subalgebra $M$ of $\operatorname{Mat}_{X}(\mathbb{C})$. We call $M$ the \emph{Bose-Mesner algebra of $\Gamma$}. Note that $M$ is generated by $A$. 

By the \emph{eigenvalues of $\Gamma$} we mean the eigenvalues of $A$. 
Note that all these eigenvalues are real.
Since $\Gamma$ is distance-regular, by \cite[p.~128]{BCN}, $\Gamma$ has $D+1$ distinct eigenvalues; we denote these eigenvalues by
\begin{equation*}
    \theta_0>\theta_1>\cdots>\theta_D.
\end{equation*}
Moreover, $\theta_0=k$ \cite[p.~129]{BCN}.

For $0\leq i\leq D$, define $E_i\in \text{Mat}_{X}(\mathbb{C})$ as 
\begin{equation}
\label{ea}
    E_i=\prod_{\substack{0\leq j\leq D\\j\neq i}}\frac{A-\theta_jI}{\theta_i-\theta_j}.
\end{equation}
For notational convenience, define $E_{-1}=E_{D+1}=0$. By \cite[Section~3]{NT2024arxiv}, we have $AE_i=E_iA=\theta_iE_i$ ($0\leq i\leq D$). We call $E_i$ the \emph{primitive idempotent} of $A$ associated with the eigenvalue $\theta_i$. By \cite[Section~3]{NT2024arxiv}, the following (i)--(v) hold:
\begin{enumerate}[label=(\roman*)]
    \item $E_0=\vert X\vert^{-1} J$;
    \item $I=\sum_{i=0}^{D}E_i$;
    \item $E_i^{\top}=E_i$ ($0\leq i\leq D$);
    \item $\overline{E_i}=E_i$ ($0\leq i\leq D$);
    \item $E_iE_j=\delta_{i,j} E_i$ ($0\leq i,j\leq D$).
\end{enumerate}
By \cite[Section~3]{NT2024arxiv}, the primitive idempotents $\{E_i\}_{i=0}^{D}$ form a basis for $M$. 

Let $V$ denote the $\mathbb{C}$-vector space consisting of the column vectors with rows indexed by $X$ and all entries in $\mathbb{C}$. The algebra $\text{Mat}_{X}(\mathbb{C})$ acts on $V$ by left multiplication. The $\text{Mat}_{X}(\mathbb{C})$-module $V$ is called \emph{standard}. By construction,
\begin{equation*}
    V=\sum_{i=0}^{D}E_iV,
\end{equation*}
where the sum is direct. For $0\leq i\leq D$ the vector space $E_iV$ is the eigenspace of $A$ associated with the eigenvalue $\theta_i$. By the {\it spectrum of $\Gamma$} we mean the set of ordered pairs $\bigl\{\left(\theta_{i},m_i\right)\bigr\}_{i=0}^{D}$, where $\left\{\theta_i\right\}_{i=0}^D$ are the
eigenvalues of $\Gamma$ and $m_i$ the dimension of the $\theta_i$-eigenspace ($0\leq i\leq D$). By \cite[Proposition~3.1]{Biggs} and the assumption that $\Gamma$ is connected, $m_0=1$.

By \cite[Section~3]{NT2024arxiv}, there exist scalars $q^{h}_{i,j}\in \mathbb{C}$ ($0\leq h,i,j\leq D$) such that
\begin{equation*}
    E_i\circ E_j=\vert X\vert ^{-1}\sum_{h=0}^{D}q^{h}_{i,j}E_h\qquad \qquad (0\leq i,j\leq D),
\end{equation*}
where $\circ$ denotes the entry-wise matrix multiplication. The scalars $q^{h}_{i,j}$ are called the \emph{Krein parameters of $\Gamma$}. 

The graph $\Gamma$ is said to be \emph{$Q$-polynomial} whenever 
\begin{equation*}
    q^{h}_{i,j}=\begin{cases}
        0&\text{if one of $h,i,j$ is greater than the sum of the other two,}\\
        \neq 0& \text{if one of $h,i,j$ is equal to the sum of the other two.}
    \end{cases}
\end{equation*}
From now on, we assume that $\Gamma$ is $Q$-polynomial.

Pick $x\in X$. For $0\leq i\leq D$, define the diagonal matrix $E^*_{i}=E^*_{i}(x)\in \text{Mat}_{X}(\mathbb{C})$ as follows. For $y\in X$, the $(y,y)$-entry of $E^*_{i}$ is
\begin{equation*}
    (E^*_{i})_{y,y}=\begin{cases}
        1&\text{if }\partial(x,y)=i,\\
        0&\text{if }\partial(x,y)\neq i.
    \end{cases}
\end{equation*}

For notational convenience, define $E^{*}_{-1}=E^*_{D+1}=0$. By \cite[Section~3]{NT2024arxiv}, we have $I=\sum_{i=0}^{D}E_i^{*}$ and $E_i^*E_j^*=\delta_{i,j}E_i^*$ ($0\leq i,j\leq D$). Hence, the matrices $\{E^*_{i}\}_{i=0}^{D}$ form a basis for a commutative subalgebra $M^{*}=M^{*}(x)$ of $\text{Mat}_{X}(\mathbb{C})$. We call $M^*$ the \emph{dual Bose-Mesner algebra} of $\Gamma$ with respect to $x$. 

Define the diagonal matrix $A^*=A^*(x)\in \text{Mat}_{X}(\mathbb{C})$ as follows. For $y\in X$, the $(y,y)$-entry of $A^*$ is
\begin{equation*}
    (A^*)_{y,y}=\vert X \vert (E_1)_{x,y}.
\end{equation*}

Since $\Gamma$ is $Q$-polynomial, by \cite[Lemma~3.11]{Ter1992JOACI} the matrix $A^*$ generates $M^*$.
Since $\{E_i^*\}_{i=0}^D$ is a basis for $M^*$, there exist scalars $\{ \theta_i^* \}_{i=0}^D$ such that
\begin{equation*}
	A^* = \sum_{i=0}^D \theta_i^* E_i^*.
\end{equation*}
By \cite[Lemma~3.11]{Ter1992JOACI}, the scalars $\{\theta_i^*\}_{i=0}^D$ are real and mutually distinct.
For $0 \leq i \leq D$, we have $A^*E_i^*=E_i^*A^*=\theta_i^*E_i^*$, and $E_i^*V$ is the eigenspace of $A^*$ corresponding to the eigenvalue $\theta_i^*$.
For $0\leq i\leq D$, we call $\theta_i^{*}$ the \emph{dual eigenvalue of $\Gamma$ associated with $E^{*}_i$}.

Let $T=T(x)$ denote the subalgebra of $\text{Mat}_{X}(\mathbb{C})$ generated by $M$ and $M^*$. The algebra $T$ is called the \emph{Terwilliger algebra of $\Gamma$ with respect to $x$}. Recall the standard module $V$. A \emph{$T$-module} is a subspace $W$ of $V$ such that $TW\subseteq W$. We say that a $T$-module $W$ is \emph{irreducible} whenever $W$ does not contain a $T$-submodule other than $0$ and itself. Since $T$ is closed under the conjugate-transpose map, any $T$-module decomposes into a direct sum of irreducible $T$-modules. In particular, the standard module $V$ is a direct sum of irreducible $T$-modules.

Let $W$ denote an irreducible $T$-module. By the \emph{endpoint $r$ of $W$}, we mean 
\begin{equation*}
    r=\min\{i\mid 0\leq i\leq D,\; E^*_{i}W\neq 0\}.
\end{equation*}

By the \emph{dual endpoint $t$ of $W$}, we mean 
\begin{equation*}
    t=\min\{i\mid 0\leq i\leq D,\; E_{i}W\neq 0\}.
\end{equation*}

By the \emph{diameter $d$ of $W$}, we mean 
\begin{equation*}
    d=\bigl\vert \{i\mid 0\leq i\leq D,\; E^*_{i}W\neq 0\}\bigr\vert-1.
\end{equation*}

By the \emph{dual diameter $\delta$ of $W$}, we mean 
\begin{equation*}
    \delta=\bigl\vert \{i\mid 0\leq i\leq D,\; E_{i}W\neq 0\}\bigr\vert-1.
\end{equation*}

By \cite[Lemma~5.5]{Ter1993JOACIII}, we have $d=\delta$.

\begin{definition}
    Let $W$ denote an irreducible $T$-module. We say that $W$ is \emph{thin} whenever $\dim E_i^* W\leq 1$ for $0\leq i\leq D$. We say that $\Gamma$ is thin whenever every irreducible $T$-module of $\Gamma$ is thin.
\end{definition}

\begin{lemma}
\label{thin}
    Let $W$ denote a thin irreducible $T$-module with endpoint $r$ and diameter $d$. Then the following (i),(ii) hold:
    \begin{enumerate}[label=(\roman*)]
        \item {\rm{\cite[Lemma~3.9]{Ter1992JOACI}}} $\dim E^*_{i}W=1$ if and only if $r\leq i\leq r+d$;
        \item {\rm{\cite[Section~2]{LIW2020LAA}}} $AE^*_{i}W\subseteq E^*_{i-1}W+E^*_{i}W+E^*_{i+1}W$, where the sum is direct.
    \end{enumerate}
\end{lemma}
 
\begin{lemma}{\rm{\cite[Definition~2.10,~Lemma~2.11]{LIW2020LAA}}}
    Let $W$ denote a thin irreducible $T$-module with endpoint $r$ and diameter $d$. In view of Lemma \ref{thin}, there exists a basis \begin{equation}
    \label{basis}
    \{w_i\in E^*_{r+i}W\mid 0\leq i\leq d\} 
    \end{equation}
    for $W$ and scalars
    \begin{equation*}
    b_i(W) \quad (0\leq i<d),\qquad \qquad a_i(W) \quad (0\leq i\leq d), \qquad \qquad c_i(W) \quad (0<i\leq d)
\end{equation*}
such that for $0\leq i\leq d$,
\begin{align*}
    &Aw_i=c_{i+1}(W)w_{i+1}+a_{i}(W)w_{i}+ b_{i-1}(W)w_{i-1},\qquad \qquad a_i(W)+b_i(W)+c_i(W)=\theta_t,
\end{align*}
where $b_{-1}(W)=b_{d}(W)=0$ and $c_{0}(W)=c_{d+1}(W)=0$.
\end{lemma}
The basis (\ref{basis}) is called \emph{standard}. The scalars $b_i(W), a_i(W), c_i(W)$ are called the \emph{intersection numbers of $W$}.

\begin{lemma}{\rm{\cite[Lemma~5.9(i)]{Cerzo}}}
\label{EAE}
    Let $W$ denote a thin irreducible $T$-module with endpoint $r$ and diameter $d$. Then the following holds on $W$:
    \begin{equation*}
        E_{r+i}^*AE_{r+i}^*=a_i(W)E^*_{r+i}\qquad \qquad (0\leq i\leq d).
    \end{equation*}
\end{lemma}

\begin{lemma}{\rm{\cite[Lemma~10.3]{Ter2024arxiv}}}
\label{acomparison}
    Let $W$ denote a thin irreducible $T$-module with endpoint $r$ and diameter $d$. Then for $0\leq i\leq d$, 
    \begin{equation*}
        a_i(W)\leq a_{r+i}.
    \end{equation*}
\end{lemma}

\section{The nucleus of $\Gamma$}
We continue to discuss the $Q$-polynomial distance-regular graph $\Gamma$.
Fix $x\in X$ and write $T=T(x)$.
In this section we define the nucleus of $\Gamma$ with respect to $x$. We recall from \cite{Ter2024arxiv} some facts about the nucleus.

\begin{definition}
    Let $W$ denote an irreducible $T$-module with endpoint $r$, dual endpoint $t$, and diameter $d$. By the \emph{displacement of $W$}, we mean the integer $r+t-D+d$.
\end{definition}

\begin{lemma}{\rm{\cite[Proposition~6.5]{Ter2024arxiv}}}
     The displacement of any irreducible $T$-module is nonnegative.
\end{lemma}

\begin{lemma}{\rm{\cite[Proposition~6.10]{Ter2024arxiv}}}
    Let $W$ denote an irreducible $T$-module with endpoint $r$, dual endpoint $t$, diameter $d$ such that the displacement is equal to $0$. Then 
\begin{equation}
\label{r=t}
    r=t,\qquad \qquad D=d+2r.
\end{equation}
\end{lemma}

\begin{definition}
    By the \emph{nucleus of $\Gamma$ with respect to $x$}, we mean the span of the irreducible $T$-modules that have displacement $0$. 
\end{definition}

Note that the nucleus is a $T$-module. Next we describe the nucleus from another point of view.

\begin{definition}
\label{Ni}
    For $0\leq i\leq D$, define the subspace $\mathcal{N}_i=\mathcal{N}_{i}(x)$ to be
    \begin{equation*}
        (E^*_{0}V+E^*_{1}V+\cdots+E^*_{i}V)\cap (E_0V+E_1V+\cdots +E_{D-i}V).
    \end{equation*}
\end{definition}

\begin{lemma}{\rm{\cite[Lemma~7.3]{Ter2024arxiv}}}
    Referring to Definition \ref{Ni}, the sum
    \begin{equation*}
    \mathcal{N}_{0}+\mathcal{N}_{1}+\cdots +\mathcal{N}_{D}
    \end{equation*}
    is direct. 
\end{lemma}
\begin{lemma}{\rm{\cite[Theorem~7.5]{Ter2024arxiv}}}
\label{defN}
    Define the subspace $\mathcal{N}=\mathcal{N}(x)$ to be
    \begin{equation*}
        \mathcal{N}_{0}+\mathcal{N}_{1}+\cdots +\mathcal{N}_{D}.
    \end{equation*}
    The subspace $\mathcal{N}$ is the nucleus of $\Gamma$ with respect to $x$.
\end{lemma}

\begin{prop}{\rm{\cite[Proposition~6.11]{Ter2024arxiv}}}
\label{endpoint}
    Let $W,W'$ denote irreducible $T$-submodules of $\mathcal{N}$. Then the following are equivalent:
    \begin{enumerate}[label=(\roman*)]
        \item the endpoints of $W,W'$ are equal;
        \item the $T$-modules $W,W'$ are isomorphic.
    \end{enumerate}
\end{prop}

In view of Proposition \ref{endpoint}, we remark that the irreducible $T$-submodules of $\mathcal{N}$ are uniquely determined, up to isomorphism of $T$-modules, by their endpoints.

By \cite[Lemma~5.2]{Ter2024arxiv}, the nucleus $\mathcal{N}$ is an orthogonal direct sum of irreducible $T$-modules. Write
\begin{equation}
\label{orthogonal}
    \mathcal{N}=\sum_{W}W\qquad \qquad (\text{orthogonal direct sum}).
\end{equation}

\begin{definition}
\label{multrdef}
    For $0\leq r\leq \frac{D}{2}$, let $\mult_r$ denote the number of summands $W$ in (\ref{orthogonal}) that have endpoint $r$.
\end{definition}

\begin{prop}{\rm{\cite[Proposition~7.11]{Ter2024arxiv}}}
\label{dimension}
    For $0\leq i\leq \frac{D}{2}$ the following subspaces have dimension $\sum_{r=0}^{i}\mult_r$:
    \begin{equation*}
        E_i\mathcal{N}, \qquad \qquad E_{D-i}\mathcal{N}, \qquad \qquad E_i^{*}\mathcal{N}, \qquad \qquad E_{D-i}^{*}\mathcal{N}, \qquad \qquad \mathcal{N}_i,\qquad \qquad \mathcal{N}_{D-i}.
    \end{equation*}
\end{prop}

\section{Projective geometry $P$ and the Grassmann graph $J_q(N,D)$}
In this section we define the projective geometry $P$ and the Grassmann graph $J_q(N,D)$. We recall several facts related to $P$ and the graph $J_q(N,D)$.

\begin{definition}
    Let $\mathbb{F}_q$ denote a finite field with $q$ elements. Let $N\geq 1$. Let $\mathcal{V}$ denote an $N$-dimensional vector space over $\mathbb{F}_q$. Let the set $P=P_q(N)$ consist of the subspaces of $\mathcal{V}$. The set $P$, together with the inclusion partial order, is a poset called a \emph{projective geometry}. 
\end{definition}

For $0\leq \ell\leq N$, let the set $P_{\ell}$ consist of the $\ell$-dimensional subspaces of $\mathcal{V}$. We have a partition
    \begin{equation}
        \label{partition1}
        P=\bigcup_{\ell=0}^{N}P_{\ell}.
    \end{equation} 
For $u,v\in P$, we say that \emph{$v$ covers $u$} whenever $u\subseteq v$ and $\dim v=\dim u+1$.

Throughout this paper, we use the notation
\begin{equation*}
[\ell] = \frac{q^{\ell}-1}{q-1} \qquad \qquad (\ell \in \mathbb{Z}).
\end{equation*}
For integers $0 \leq n \leq \ell$, we denote the $q$-binomial coefficient by
\begin{equation}
\label{qbinomdef}
\binom{\ell}{n}_{q} = \frac{[\ell][\ell-1]\cdots[\ell-n+1]}{[n][n-1]\cdots[2][1]}.
\end{equation}
For convenience, we set $\binom{\ell}{n}_q = 0$ whenever $\ell < n$ or $n < 0$.

\begin{lemma}{\rm{\cite[Theorem~9.3.2]{BCN}}}
\label{qbinom}
For $0 \leq n \leq \ell$ and $u \in P_\ell$, the number of $n$-dimensional subspaces contained in $u$ is equal to $\binom{\ell}{n}_q$.
\end{lemma}

\begin{corollary}
\label{cover}
    For $u\in P$, the number of subspaces covered by $u$ is equal to $[\ell]$, where $\ell=\dim u$.
\end{corollary}

\begin{proof}
    Set $n=\ell-1$ in Lemma \ref{qbinom} and use (\ref{qbinomdef}).
\end{proof}

Next we refine the partition of $P$ in (\ref{partition1}). Assume that $N > D \geq 1$, and pick $x \in P_D$. 
For $0 \leq i\leq D$ and $0 \leq j\leq N-D$ define
    \begin{equation*}
        P_{i,j}=\bigl\{u\in P\mid \dim (u\cap x)=i, \dim u=i+j\bigr\}.
    \end{equation*}

    For integers $\ell,n$ we define $P_{\ell,n}$ to be empty unless $0\leq \ell\leq D$ and $0\leq n\leq N-D$.

    We have a partition
    \begin{equation}
    \label{partition2}
        P=\bigcup_{i=0}^{D}\bigcup_{j=0}^{N-D}P_{i,j}.
    \end{equation}

    In the diagram below, we illustrate the set $P$ and the subsets $P_{i,j}$ $(0\leq i\leq D,\; 0\leq j\leq N-D)$.
    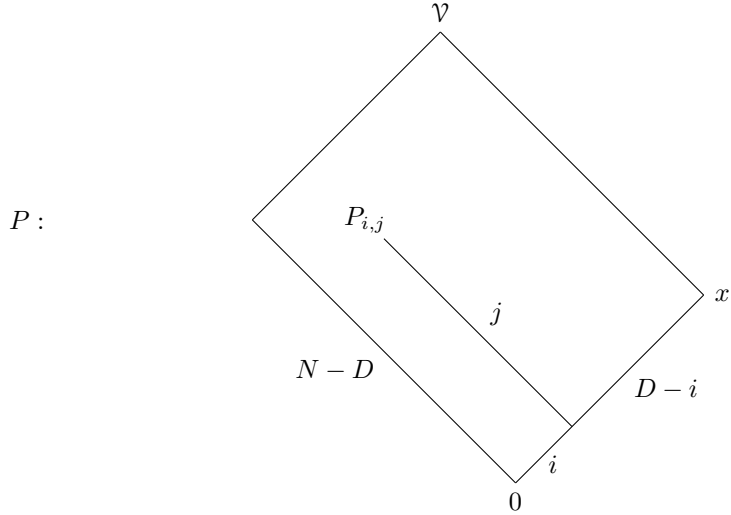
\begin{figure}[!ht]
\centering
{%
\begin{circuitikz}
\tikzstyle{every node}=[font=\normalsize]
\draw [short] (2.5,-3.5) -- (-1,0);
\draw [short] (-1,0) -- (1.5,2.5);
\draw [short] (1.5,2.5) -- (5,-1);
\draw [short] (5,-1) -- (2.5,-3.5);
\draw [short] (0.75,-0.25) -- (3.25,-2.75);
\node [font=\normalsize] at (1.5,2.75) {$\mathcal{V}$};
\node [font=\normalsize] at (5.25,-1) {$x$};
\node [font=\normalsize] at (2.5,-3.75) {$0$};
\node [font=\normalsize] at (3,-3.25) {$i$};
\node [font=\normalsize] at (2.25,-1.25) {$j$};
\node [font=\normalsize] at (0.5,0) {$P_{i,j}$};
\node [font=\normalsize] at (4.5,-2.25) {$D-i$};
\node [font=\normalsize] at (0.1,-2) {$N-D$};
\node [font=\normalsize] at (-4,0) {$P:$};
\node [font=\normalsize] at (8,0) {$\;$};
\end{circuitikz}
}%
\caption{The projective geometry $P$ and the location of $P_{i,j}$.}
\label{projective}
\end{figure}

Below \eqref{partition1}, we described the covering relation on $P$. Next, we give a refinement of the covering relation.
    
    \begin{lemma}{\rm{\cite[Lemma~2.3]{Seong2024arxiv}}}
    \label{coverlem}
        Let $u,v\in P$ such that $v$ covers $u$. Write
        \begin{equation*}
            u\in P_{i,j}, \qquad \qquad v\in P_{\ell,n}.
        \end{equation*}
        Then either (i) $\ell=i+1$ and $n=j$, or (ii) $\ell=i$ and $n=j+1$.
    \end{lemma}
        
    \begin{definition}
    \label{slashcover}
        Referring to Lemma \ref{coverlem},  we say that \emph{$v$ $\slash$-covers $u$} whenever (i) holds, and \emph{$v$ $\backslash$-covers $u$} whenever (ii) holds.
    \end{definition}

    We illustrate Definition \ref{slashcover} using the diagrams below.
    \begin{figure}[!ht]
\centering
{%
\begin{circuitikz}
\tikzstyle{every node}=[font=\normalsize]
\draw [short] (2.5,-3.5) -- (-1,0);
\draw [short] (-1,0) -- (1.5,2.5);
\draw [short] (1.5,2.5) -- (5,-1);
\draw [short] (5,-1) -- (2.5,-3.5);
\draw [short] (0.9,-0.1) -- (1.1,0.1);
\node [font=\normalsize] at (1.5,2.75) {$\mathcal{V}$};
\node [font=\normalsize] at (5.25,-1) {$x$};
\node [font=\normalsize] at (2.5,-3.75) {$0$};
\node [font=\normalsize] at (1.25,0.25) {$v$};
\node [font=\normalsize] at (0.75,-0.25) {$u$};
\node [font=\normalsize] at (4,-2.5) {$D$};
\node [font=\normalsize] at (0.1,-2) {$N-D$};
\node [font=\normalsize] at (2,-4.75) {Figure 2.1: $v$ $\slash$-covers $u$.};
\end{circuitikz}
\hspace{0.25cm}
\begin{circuitikz}
\draw [short] (9.5,-3.5) -- (6,0);
\draw [short] (6,0) -- (8.5,2.5);
\draw [short] (8.5,2.5) -- (12,-1);
\draw [short] (12,-1) -- (9.5,-3.5);
\draw [short] (8.4,0.1) -- (8.6,-0.1);

\node [font=\normalsize] at (8.5,2.75) {$\mathcal{V}$};
\node [font=\normalsize] at (12.25,-1) {$x$};
\node [font=\normalsize] at (9.5,-3.75) {$0$};
\node [font=\normalsize] at (8.25,0.25) {$v$};
\node [font=\normalsize] at (8.75,-0.25) {$u$};
\node [font=\normalsize] at (11,-2.5) {$D$};
\node [font=\normalsize] at (7.1,-2) {$N-D$};
\node [font=\normalsize] at (9,-4.75) {Figure 2.2: $v$ $\backslash$-covers $u$.};
\end{circuitikz}
}%
\end{figure}
\renewcommand{\thefigure}{\arabic{figure}}
\newpage
\setcounter{figure}{2}

Next we define the Grassmann graph $J_q(N,D)$.

\begin{definition}
    The Grassmann graph $J_q(N,D)$ has vertex set $P_D$. Two vertices $y,z$ form an edge whenever $y\cap z\in P_{D-1}$.
\end{definition}

For the rest of this paper, we set $\Gamma=J_q(N,D)$. By \cite[p.~268]{BCN}, the graph $\Gamma$ is isomorphic to $J_q(N,N-D)$. Without loss, we assume that $N\geq 2D$. The case $N=2D$ is special, and we discuss it briefly in Section \ref{remark}. Unless otherwise stated, we assume $N>2D$. Under this assumption, the diameter of $\Gamma$ is $D$ (see \cite[Theorem~9.3.3]{BCN}).

By \cite[Theorem~9.3.2]{BCN}, the valency of $\Gamma$ is 
\begin{equation}
\label{kappa}
  k=q[D][N-D].  
\end{equation}

By \cite[Theorem~9.3.3]{BCN}, the intersection numbers of $\Gamma$ are
\begin{equation}
    \label{sizebc}
    b_i=q^{2i+1}[D-i][N-D-i],\qquad \qquad c_i=[i]^2\qquad \qquad (0\leq i\leq D).
\end{equation}

By \cite[Theorem~9.3.3]{BCN}, the eigenvalues of $\Gamma$ are
 \begin{equation}
 \label{eigenvalues}
     \theta_i=q[D][N-D]-[i][N-i+1] \qquad \qquad (0\leq i\leq D).
 \end{equation}

 By \cite[Section~3]{Lee2020LAA}, the dual eigenvalues of $\Gamma$ are
\begin{equation}
\label{dual}
    \theta_{i}^{*}=-\frac{q[N-1]\bigl([D]+[N-D]\bigr)}{(q-1)[D][N-D]}+\frac{q[N][N-1]}{(q-1)[D][N-D]}q^{-i}\qquad \qquad (0\leq i\leq D).
\end{equation}

\section{The algebra $\mathcal{H}$}
Recall the projective geometry $P$. Let $\text{Mat}_{P}(\mathbb{C})$ denote the $\mathbb{C}$-algebra consisting of the matrices with rows and columns indexed by $P$ and all entries in $\mathbb{C}$. We continue to fix $x\in P_D$. In this section we recall from \cite[Section~7]{Watanabe} the subalgebra $\mathcal{H}=\mathcal{H}(x)$ of $\text{Mat}_{P}(\mathbb{C})$.

\begin{definition}
        For $0 \leq i\leq D$ and $0 \leq j\leq N-D$, define the diagonal matrix $E^{*}_{i,j}\in \text{Mat}_P(\mathbb{C})$ as follows. For $u\in P$, the $(u,u)$-entry of $E^{*}_{i,j}$ is
    \begin{equation*}
        \bigl(E^{*}_{i,j}\bigr)_{u,u}=\begin{cases}
            1&\text{if }u\in P_{i,j},\\
            0&\text{if $u\notin P_{i,j}$}.
        \end{cases}
    \end{equation*}  
    For notational convenience, for integers $\ell,n$ we define $E^{*}_{\ell,n}=0$ unless $0 \leq \ell\leq D$ and $0 \leq n\leq N-D$.
    \end{definition}

    \begin{definition}{\cite[Definition~6.1]{Watanabe}}
        The matrices 
        \begin{equation*}
            E^{*}_{i,j}\qquad \qquad 0\leq i\leq D\qquad \qquad 0\leq j\leq N-D
        \end{equation*} 
        form a basis for a commutative subalgebra of $\text{Mat}_{P}(\mathbb{C})$. Denote this subalgebra by $\mathcal{K}$.
    \end{definition}

    Next we find a generating set for $\mathcal{K}$. For the rest of the paper, $q^{1/2}$ denotes the positive square root of $q$.

\begin{definition}
    Define diagonal matrices $K_1,K_2\in \text{Mat}_P(\mathbb{C})$ as follows. For $u\in P$ their $(u,u)$-entries are
    \begin{equation*}
        (K_1)_{u,u}=q^{\frac{D}{2}-i},\qquad \qquad (K_2)_{u,u}=q^{\frac{N-D}{2}-j},
    \end{equation*}
    where $u\in P_{i,j}$.

    Note that $K_1,K_2$ are invertible. By \cite[Proposition~6.3]{Watanabe}, the algebra $\mathcal{K}$ is generated by $K_1^{\pm 1},K_2^{\pm 1}$.
\end{definition}

\begin{definition}
    Define matrices $L_1,L_2,R_1,R_2\in \text{Mat}_P(\mathbb{C})$ as follows. For $u,v\in P$ their $(u,v)$-entries are
    \begin{align*}
        \bigl(L_1\bigr)_{u,v}&=\begin{cases}
            1&\text{if }v\text{ $\slash$-covers }u,\\
            0&\text{if }v\text{ does not $\slash$-cover }u,
        \end{cases}\\
        \bigl(L_2\bigr)_{u,v}&=\begin{cases}
            1&\text{if }v\text{ $\backslash$-covers }u,\\
            0&\text{if }v\text{ does not $\backslash$-cover }u,
        \end{cases}\\
        \bigl(R_1\bigr)_{u,v}&=\begin{cases}
            1&\text{if }u\text{ $\slash$-covers }v,\\
            0&\text{if }u\text{ does not $\slash$-cover }v,
        \end{cases}\\
        \bigl(R_2\bigr)_{u,v}&=\begin{cases}
            1&\text{if }u\text{ $\backslash$-covers }v,\\
            0&\text{if }u\text{ does not $\backslash$-cover }v.
        \end{cases}
    \end{align*}
    Note that $R_1=L_1^{\top}$ and $R_2=L_2^{\top}$. 
\end{definition}

\begin{definition}
    Let $\mathcal{H}=\mathcal{H}(x)$ denote the subalgebra of $\text{Mat}_{P}(\mathbb{C})$ generated by $L_1,L_2,R_1,R_2, K_1^{\pm1}, K_2^{\pm 1}$.
\end{definition}

Let $\Psi$ denote the $\mathbb{C}$-vector space consisting of the column vectors with rows indexed by $P$ and all entries in $\mathbb{C}$. The algebra $\operatorname{Mat}_{P}(\mathbb{C})$ acts on $\Psi$ by left multiplication. The $\operatorname{Mat}_{P}(\mathbb{C})$-module $\Psi$ is called \emph{standard}. By construction, the vector space $\Psi$ is an $\mathcal{H}$-module. Let $\Omega$ denote an $\mathcal{H}$-submodule of $\Psi$. We say that $\Omega$ is \emph{irreducible} whenever $\Omega$ does not contain an $\mathcal{H}$-submodule besides $0$ or $\Omega$. Since $\mathcal{H}$ is closed under the conjugate-transpose map, every $\mathcal{H}$-module decomposes into a direct sum of irreducible $\mathcal{H}$-modules. In particular, the standard module $\Psi$ is a direct sum of irreducible $\mathcal{H}$-modules.

\begin{lemma}{\rm{\cite[Lemma~2.13]{Seong2024arxiv}}}
    \label{typelem}
        Let $\Omega$ denote an irreducible $\mathcal{H}$-module. Then there exist integers $\alpha,\beta,\rho$ such that the following (i), (ii) hold:
        \begin{enumerate}[label=(\roman*)]
            \item 
            \begin{equation}
            \label{abcondition}
                0\leq \rho,\qquad \qquad 0\leq \alpha\leq \frac{D-\rho}{2}, \qquad \qquad 0\leq \beta\leq \frac{N-D-\rho}{2};
            \end{equation}

            \item for $0\leq i\leq D$ and $0\leq j\leq N-D$,
            \begin{equation*}
            \dim E^{*}_{i,j}\Omega=\begin{cases}
                1&\text{if }\alpha\leq i\leq D-\rho-\alpha,\;\rho+\beta\leq j\leq N-D-\beta,\\
                0&\text{otherwise}.
            \end{cases}
        \end{equation*}
        \end{enumerate}
    \end{lemma}

    We illustrate Lemma \ref{typelem} using the diagram below. 

    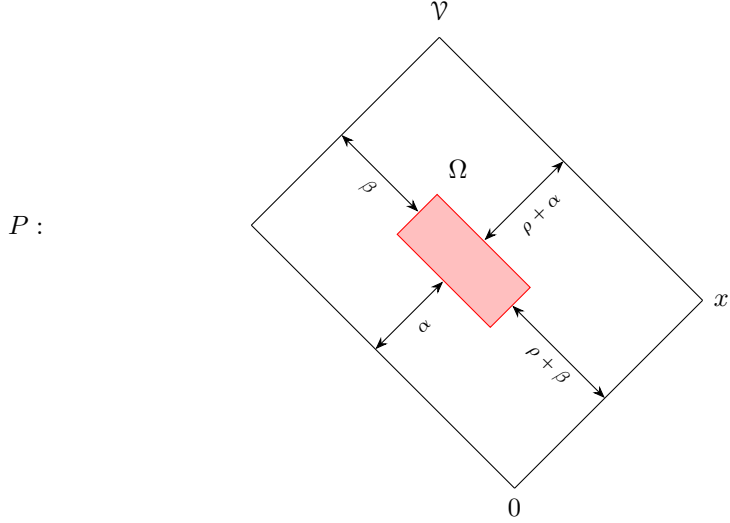
\begin{figure}[!ht]
\centering
{%
\begin{circuitikz}
\tikzstyle{every node}=[font=\normalsize]
\draw [short] (2.5,-3.5) -- (-1,0);
\draw [short] (-1,0) -- (1.5,2.5);
\draw [short] (1.5,2.5) -- (5,-1);
\draw [short] (5,-1) -- (2.5,-3.5);
\draw [ color=red , fill=pink, rotate around={45:(2,2)}] (0.5,1.25) rectangle (-0.25,-0.5);
\draw [<->, >=Stealth] (2.1,-0.2) -- (3.15,0.85);
\draw [<->, >=Stealth] (0.2,1.2) -- (1.22,0.18);
\draw [<->, >=Stealth] (3.7,-2.3) -- (2.47,-1.07);
\draw [<->, >=Stealth] (1.55,-0.75) -- (0.65,-1.65);
\node [font=\normalsize] at (1.5,2.85) {$\mathcal{V}$};
\node [font=\normalsize] at (5.25,-1) {$x$};
\node [font=\normalsize] at (2.5,-3.75) {$0$};
\node [font=\normalsize] at (1.75,0.75) {$\Omega$};
\node [font=\normalsize] at (-4,0) {$P:$};
\node [font=\normalsize] at (8,0) {$\;$};
\node [font=\scriptsize, rotate around={45:(0,0)}] at (1.3,-1.35) {$\alpha$};
\node [font=\scriptsize, rotate around={45:(0,0)}] at (2.85,0.15) {$\rho+\alpha$};
\node [font=\scriptsize, rotate around={-45:(0,0)}] at (2.95,-1.85) {$\rho+\beta$};
\node [font=\scriptsize, rotate around={-45:(0,0)}] at (0.55,0.5) {$\beta$};

\end{circuitikz}
}%
\caption{The irreducible $\mathcal{H}$-module $\Omega$ is represented by the red rectangle.}
\label{locationabp}
\end{figure}
\newpage

\begin{definition}
        Let $\Omega$ denote an irreducible $\mathcal{H}$-module. Referring to Lemma \ref{typelem}, we call the triple $(\alpha,\beta,\rho)$ the \emph{type of }$\Omega$. 
    \end{definition}

    \begin{lemma}{\rm{\cite[p.~133]{LIW2020LAA}}}
    \label{unique}
        For each triple $(\alpha,\beta,\rho)$ of integers that satisfy (\ref{abcondition}), there exists an irreducible $\mathcal{H}$-module of type $(\alpha,\beta,\rho)$ that is unique up to isomorphism.
    \end{lemma}

\begin{definition}
By the \emph{primary $\mathcal{H}$-module}, we mean the irreducible $\mathcal{H}$-module of type $(0,0,0)$; denote this module by $\Omega_0$.
\end{definition}

Recall that the standard module $\Psi$ is a direct sum of irreducible $\mathcal{H}$-modules. Write 
\begin{equation}
    \label{direct}
        \Psi=\sum_{\Omega}\Omega\qquad \qquad (\text{direct sum}),
    \end{equation}
    where $\Omega$ are irreducible $\mathcal{H}$-modules.

\begin{lemma} 
\label{e*ijpsi}
    For $0\leq i\leq D$ and $0\leq j\leq N-D$, the dimension of $E_{i,j}^{*}\Psi$ is equal to the number of summands $\Omega$ in (\ref{direct}) such that $E_{i,j}^{*}\Omega\neq 0$.
\end{lemma}

\begin{proof}
    Apply $E^*_{i,j}$ to both sides of (\ref{direct}) to obtain
    \begin{equation}
    \label{e*ij}
        E^*_{i,j}\Psi=\sum_{\Omega}E^*_{i,j}\Omega.
    \end{equation}

    By (\ref{e*ij}),
    \begin{equation}
    \label{dime*ij}
        \dim E^*_{i,j}\Psi=\sum_{\Omega}\dim E^*_{i,j}\Omega.
    \end{equation}
    By Lemma \ref{typelem}(ii) and the fact that the summands $\Omega$ in (\ref{direct}) are irreducible $\mathcal{H}$-modules, the dimension of $E^{*}_{i,j}\Omega$ is either $0$ or $1$ for all $\Omega$. The result follows from (\ref{dime*ij}).
\end{proof}

We now consider a specific type of irreducible $\mathcal{H}$-modules that will be useful. We say that an irreducible $\mathcal{H}$-module $\Omega$ is \emph{alpha-dominant} whenever $\Omega$ has type $(\alpha,0,0)$ for some $\alpha$ $(0\leq \alpha\leq \frac{D}{2})$. In this case, we call $\alpha$ the \emph{gap of $\Omega$}. Observe that the $\mathcal{H}$-module $\Omega_0$ is alpha-dominant with gap $0$.

\begin{lemma}
\label{alphadominant}
    For $0\leq \alpha\leq \frac{D}{2}$, let $\mu_\alpha$ denote the number of summands $\Omega$ in (\ref{direct}) that are alpha-dominant with gap $\alpha$. Then
    \begin{equation}
    \label{m_alpha}
        \mu_\alpha=\binom{D}{\alpha}_q-\binom{D}{\alpha-1}_q.
    \end{equation}
\end{lemma}

\begin{proof}  
We proceed by strong induction on $\alpha$. Suppose that $\alpha=0$. Then $\Omega=\Omega_0$, so (\ref{m_alpha}) holds. Now suppose that (\ref{m_alpha}) holds for $0\leq \alpha\leq \ell$, where $\ell$ is a nonnegative integer with $\ell<\left\lfloor \frac{D}{2}\right\rfloor$. We will show that (\ref{m_alpha}) also holds for $\alpha=\ell+1$. By construction, the dimension of $E^{*}_{\ell+1,0}\Psi$ is equal to the number of $(\ell+1)$-dimensional subspaces of $x$. By Lemma \ref{qbinom}, this number is
    \begin{equation*}
        \binom{D}{\ell+1}_q.
    \end{equation*}
Observe that an irreducible $\mathcal{H}$-module $\Omega$ satisfies $E^*_{\ell+1,0}\Omega\neq 0$ if and only if $\Omega$ is alpha-dominant with gap $\alpha$ for $0\leq \alpha\leq \ell+1$. 
By the inductive hypothesis,
$$
   \sum_{\alpha=0}^{\ell} \mu_\alpha = \binom{D}{\ell}_q,
$$
which is the total number of irreducible $\mathcal{H}$-modules that are $\alpha$-dominant with gap $\alpha$ for $0 \leq \alpha \leq \ell$.
By Lemma \ref{e*ijpsi},
    \begin{equation*}
        \mu_{\ell+1}=\binom{D}{\ell+1}_q-\binom{D}{\ell}_q.
    \end{equation*}
    The result follows.
\end{proof}

\section{Irreducible $T$-modules of $J_q(N,D)$ and its nucleus}
Recall the Grassmann graph $\Gamma=J_q(N,D)$. 
Pick $x \in X$ and write $T=T(x)$.
In this section we discuss how the irreducible $T$-modules of $\Gamma$ are related to the irreducible $\mathcal{H}$-modules. For the nucleus $\mathcal{N}$ of $\Gamma$, we find the multiplicity of the irreducible $T$-submodules of $\mathcal{N}$ with a given endpoint. 

By \cite[Example~6.1]{Ter1993JOACIII}, each irreducible $T$-module of $\Gamma$ is determined up to isomorphism by the following parameters: endpoint $r$, dual endpoint $t$, diameter $d$, auxiliary parameter $e$. Furthermore, these parameters satisfy the following (i)
--(iv):
\begin{itemize}
\setlength\itemsep{0pt}
	\item[(i)] $0 \leq \frac{D-d}{2} \leq r \leq t \leq D-d \leq D$;
	\item[(ii)] $e+d+D$ is even;
	\item[(iii)] $|e| \leq 2r - D + d$;
	\item[(iv)] $d \in \{ e + D- 2r, \min\{ D - t, e + D - 2r + 2(N-2D) \} \}$.
\end{itemize}

Let $\Delta$ denote the set of quadruples $(r,t,d,e)$ of integers that satisfy (i)--(iv) above. By \cite[Theorem~3.1(i),~5.6]{LIW2020LAA}, for each quadruple $(r,t,d,e)\in \Delta$, there exists an irreducible $T$-module of $\Gamma$ that has the parameters $r,t,d,e$.

\begin{definition}
    By the \emph{primary $T$-module of $\Gamma$}, we mean an irreducible $T$-module of $\Gamma$ with endpoint $0$, dual endpoint $0$, diameter $D$, and the auxiliary parameter $0$. Denote this module by $W_0$.
\end{definition} 

By \cite[Lemma~2.13]{LIW2020LAA}, the graph $\Gamma$ is thin. Let $W$ denote an irreducible $T$-module of $\Gamma$ with endpoint $r$, dual endpoint $t$, diameter $d$ and auxiliary parameter $e$. By \cite[Example~6.1]{Ter1993JOACIII}, the intersection numbers of $W$ are
\begin{align}
    a_i(W)&=q[D][N-D]-[t][N+1-t]-b_i(W)-c_i(W) &(0\leq i\leq d),\label{aW}\\
    b_i(W)&=q^{2i+1+r+\frac{D-d-e}{2}}[d-i]\Bigl[N-i-r-t+\frac{e-D-d}{2}\Bigr]&(0\leq i<d),\label{bW}\\
    c_i(W)&=q^{t}[i]\Bigl[i+r-t+\frac{D-d-e}{2}\Bigr]&(0<i\leq d),\label{cW}
\end{align}
where $b_d(W)=c_0(W)=0$.

Recall the standard $\text{Mat}_{P}(\mathbb{C})$-module $\Psi$ and the standard $\text{Mat}_{X}(\mathbb{C})$-module $V$. We may regard $V$ as a subspace of $\Psi$ in the sense that $V$ is isomorphic to 
\begin{equation*}
    \sum_{j=0}^{D}E^*_{D-j,j}\Psi \qquad \qquad (\text{direct sum})
\end{equation*}
as vector spaces. We use the following notation. For an $\mathcal{H}$-submodule $\Omega$ of $\Psi$, let $\widetilde{\Omega}$ denote the subspace of $V$ defined by the restriction of 
\begin{equation*}
    \sum_{j=0}^{D}E^*_{D-j,j}\Omega
\end{equation*}
to $V$. 

\begin{lemma}{\rm{\cite[Proposition~5.5]{LIW2020LAA}}}
\label{tildelem}
    Let $\Omega$ denote an irreducible $\mathcal{H}$-module. Then $\widetilde{\Omega}$ is an irreducible $T$-module of $\Gamma$.
\end{lemma}

We now describe how the type of $\Omega$ converts to the parameters of $\widetilde{\Omega}$. There are three cases to consider:

\begin{enumerate}[label=(C\arabic*)]
        \item $\beta-\alpha\leq 0$;
        \item $0<\beta-\alpha\leq N-2D$;
        \item $N-2D< \beta-\alpha$.
    \end{enumerate}
    We illustrate the cases (C1)--(C3) using the diagrams below.

    \begin{figure}[!ht]
    \centering
{%
\begin{circuitikz}
\tikzstyle{every node}=[font=\normalsize]
\draw [short] (2.5,-3.5) -- (-1,0);
\draw [short] (-1,0) -- (1.5,2.5);
\draw [short] (1.5,2.5) -- (5,-1);
\draw [short] (5,-1) -- (2.5,-3.5);
\draw [ color=red , fill=pink, rotate around={45:(2,2)}] (0.5,2.25) rectangle (-0.25,-1.5);
\draw [dashed] (5,-1) -- (0, -1);
\draw [<->, >=Stealth] (2.9,-0.9) -- (4.85,-0.9);
\draw [<->, >=Stealth] (1.95,-1.1) -- (2.9,-1.1);
\draw [short] (2.8,-2.09) -- (3,-2.09);
\draw [short] (3,-3.5) -- (2.4,-3.5);
\draw [<->, >=Stealth] (2.9,-2.09) -- (2.9,-3.5);
\node [font=\scriptsize] at (3.95,-0.75) {$r$};
\node [font=\scriptsize] at (2.7,-2.8) {$t$};
\node [font=\scriptsize] at (2.45,-1.25) {$d$};
\node [font=\normalsize] at (1.5,2.85) {$\mathcal{V}$};
\node [font=\normalsize] at (5.25,-1) {$x$};
\node [font=\normalsize] at (2.5,-3.75) {$0$};
\node [font=\normalsize] at (1.75,0.75) {$\Omega$};
\node [font=\normalsize] at (2,-4.5) {Figure 4.1: Case (C1)};
\end{circuitikz}
\hspace{0.25cm}
\begin{circuitikz}
\tikzstyle{every node}=[font=\normalsize]
\draw [short] (2.5,-3.5) -- (-1,0);
\draw [short] (-1,0) -- (1.5,2.5);
\draw [short] (1.5,2.5) -- (5,-1);
\draw [short] (5,-1) -- (2.5,-3.5);
\draw [ color=red , fill=pink, rotate around={45:(2,2)}] (0.5,1.25) rectangle (-0.25,-0.5);
\draw [dashed] (5,-1) -- (0, -1);
\draw [<->, >=Stealth] (2.55,-1.1) -- (4.85,-1.1);
\draw [<->, >=Stealth] (1.75,-0.9) -- (2.55,-0.9);
\draw [short] (2,-1.35) -- (2.3,-1.35);
\draw [short] (2,-3.5) -- (2.6,-3.5);
\draw [<->, >=Stealth] (2.15,-1.35) -- (2.15,-3.5);
\node [font=\scriptsize] at (3.75,-1.3) {$r$};
\node [font=\scriptsize] at (2.35,-2.4) {$t$};
\node [font=\scriptsize] at (2.15,-0.75) {$d$};
\node [font=\normalsize] at (1.5,2.85) {$\mathcal{V}$};
\node [font=\normalsize] at (5.25,-1) {$x$};
\node [font=\normalsize] at (2.5,-3.75) {$0$};
\node [font=\normalsize] at (1.75,0.75) {$\Omega$};
\node [font=\normalsize] at (2,-4.5) {Figure 4.2: Case (C2)};
\end{circuitikz}
}%
\end{figure}

\begin{figure}[!ht]
    \centering
    {%
    \begin{circuitikz}
\tikzstyle{every node}=[font=\normalsize]
\draw [short] (2.5,-3.5) -- (-1,0);
\draw [short] (-1,0) -- (1.5,2.5);
\draw [short] (1.5,2.5) -- (5,-1);
\draw [short] (5,-1) -- (2.5,-3.5);
\draw [ color=red , fill=pink, rotate around={45:(2,2)}] (1.4,0.65) rectangle (-1.15,-0.3);
\draw [dashed] (5,-1) -- (0, -1);
\draw [<->, >=Stealth] (2.3,-1.1) -- (4.85,-1.1);
\draw [<->, >=Stealth] (1.05,-0.9) -- (2.3,-0.9);
\draw [short] (1.3,-1.85) -- (1.5,-1.85);
\draw [short] (1.3,-3.5) -- (2.6,-3.5);
\draw [<->, >=Stealth] (1.4,-1.85) -- (1.4,-3.5);
\node [font=\scriptsize] at (3.55,-1.3) {$r$};
\node [font=\scriptsize] at (1.2,-2.7) {$t$};
\node [font=\scriptsize] at (1.7,-0.75) {$d$};
\node [font=\normalsize] at (1.5,2.85) {$\mathcal{V}$};
\node [font=\normalsize] at (5.25,-1) {$x$};
\node [font=\normalsize] at (2.5,-3.75) {$0$};
\node [font=\normalsize] at (1.9,0.6) {$\Omega$};
\node [font=\normalsize] at (2,-4.5) {Figure 4.3: Case (C3)};
\end{circuitikz}
}%
\end{figure}

\renewcommand{\thefigure}{\arabic{figure}}

\setcounter{figure}{6}
\newpage
\begin{lemma}{\rm{\cite[pp.~126--127]{LIW2020LAA}}}
\label{conversion}
    Let $\Omega$ denote an irreducible $\mathcal{H}$-module of type $(\alpha,\beta,\rho)$. Below we express the parameters of $\widetilde{\Omega}$ in terms of $\alpha,\beta,\rho$. 
    
    In case (C1),
    \begin{equation}
    \label{conversion1}
        r=\rho+\alpha, \qquad \qquad t=\rho+\alpha+\beta, \qquad \qquad d=D-\rho-2\alpha, \qquad \qquad e=\rho.
    \end{equation}
    In case (C2),
    \begin{equation}
    \label{conversion2}
        r=\rho+\beta, \qquad \qquad t=\rho+\alpha+\beta, \qquad \qquad d=D-\rho-\alpha-\beta, \qquad \qquad e=\rho+\alpha-\beta.
    \end{equation}
    In case (C3),
    \begin{equation}
    \label{conversion3}
        r=\rho+\beta, \qquad \qquad t=\rho+\alpha+\beta, \qquad \qquad d=N-D-\rho-2\beta, \qquad \qquad e=\rho-N+2D.
    \end{equation}
\end{lemma}

Our next goal is to find the multiplicity of irreducible $T$-submodules of $\mathcal{N}$ with a given endpoint. We first discuss how some $\mathcal{H}$-modules are related to the nucleus $\mathcal{N}$.

\begin{remark}
\label{primary}
    Recall the primary $\mathcal{H}$-module $\Omega_0$. Since $\Omega_0$ belongs to case (C1), we have $\widetilde{\Omega_0}=W_0$  by (\ref{conversion1}). Observe that $W_0$ has displacement $0$. Hence, $W_0$ is an irreducible $T$-submodule of $\mathcal{N}$.
\end{remark}

\begin{lemma}
\label{beta-alpha}
    Let $\Omega$ denote an irreducible $\mathcal{H}$-module of type $(\alpha,\beta,\rho)$ such that $\widetilde{\Omega}$ is an irreducible $T$-submodule of the nucleus $\mathcal{N}$. Then $\beta-\alpha\leq 0$.
\end{lemma}

\begin{proof}
    By Remark \ref{primary}, there exists an irreducible $\mathcal{H}$-module $\Omega$ such that $\widetilde{\Omega}$ is an irreducible $T$-submodule of the nucleus $\mathcal{N}$. It suffices to show that $\Omega$ satisfies case (C1) only. We prove this by contradiction. Suppose that case (C2) holds. Using the left equation of (\ref{r=t}) and the first two equations of (\ref{conversion2}), we obtain $\alpha=0$. Combining the first and third equations of (\ref{conversion2}), we obtain $d=D-r$. By the right equation of (\ref{r=t}), we have $d=D-2r$. Combining the comments above, we obtain $r=0$. Combining (\ref{abcondition}) and the first equation of (\ref{conversion2}), we have $\rho=0$ and $\beta=0$. This contradicts the condition for case (C2).

    Next suppose that case (C3) holds. Using the left equation of (\ref{r=t}) and the first two equations of (\ref{conversion3}), we obtain $\alpha=0$. By the right equation of (\ref{r=t}) we have $d=D-2\rho-2\beta$. Combining the previous equation and the third equation of (\ref{conversion3}), we have
    \begin{equation*}
        N-2D+\rho=0.
    \end{equation*}
    Since $N-2D>0$ by assumption and $\rho\geq 0$ by (\ref{abcondition}), we have a contradiction. Therefore, $\Omega$ satisfies case (C1) only. The result follows.
\end{proof}

\begin{lemma}
\label{DominantEquiv}
    Let $\Omega$ denote an irreducible $\mathcal{H}$-module. Then the following are equivalent:
    \begin{enumerate}[label=(\roman*)]
        \item $\Omega$ is alpha-dominant with gap $\alpha$;
        \item $\widetilde{\Omega}$ is an irreducible $T$-submodule of $\mathcal{N}$ with endpoint $\alpha$.
    \end{enumerate}
\end{lemma}

\begin{proof}
    (i)$\Rightarrow$(ii) The fact that $\widetilde{\Omega}$ is an irreducible $T$-module follows from Lemma \ref{tildelem}. By assumption, the $\mathcal{H}$-module $\Omega$ has type $(\alpha,0,0)$. Note that $\Omega$ satisfies case (C1). Using the conversion in (\ref{conversion1}), we observe that the parameters of $\widetilde{\Omega}$ satisfy (\ref{r=t}). The result follows from the definition of $\mathcal{N}$ and the first equation of (\ref{conversion1}).

    (ii)$\Rightarrow$(i) By (\ref{r=t}), the $T$-module $\tilde{\Omega}$ has endpoint $\alpha$, dual endpoint $\alpha$, diameter $D-2\alpha$. By Lemma \ref{beta-alpha}, the $\mathcal{H}$-module $\Omega$ satisfies case (C1). By (\ref{conversion1}), the $\mathcal{H}$-module $\Omega$ has type $(\alpha,0,0)$. The result follows. 
\end{proof}

We use the following notation. 
For a $T$-module $W$, let $W^{\mathcal{H}}$ denote the subspace of $\Psi$ defined as the intersection of all $\mathcal{H}$-submodules $\Omega \subseteq \Psi$ with $\widetilde{\Omega} = W$.
Observe that $W^{\mathcal{H}}$ is an $\mathcal{H}$-module.

\begin{lemma}
\label{HW}
    Let $W$ denote an irreducible $T$-submodule of the nucleus $\mathcal{N}$ with endpoint $r$. Then $W^{\mathcal{H}}$ is an irreducible $\mathcal{H}$-module that is alpha-dominant with gap $r$.
\end{lemma}

\begin{proof}
    We first show that $W^{\mathcal{H}}$ is an irreducible $\mathcal{H}$-module. Since $W$ has displacement $0$, we have $0 \leq r \leq \frac{D}{2}$. We first assume that $r<\frac{D}{2}$. By the right equation of (\ref{r=t}), we have $d\geq 1$. The irreducibility follows from \cite[Theorem~5.7]{LIW2020LAA} and the assumption that $N>2D$. 

    Next we assume that $r=\frac{D}{2}$. By the right equation of (\ref{r=t}), we have $d=0$. The irreducibility follows from \cite[Theorem~5.8]{LIW2020LAA} and the assumption that $N>2D$. 
    
    Next we show that $W^{\mathcal{H}}$ is alpha-dominant with gap $r$. Observe that $\widetilde{W^{\mathcal{H}}}=W$. The result follows from Lemma \ref{DominantEquiv}.  
\end{proof}

\begin{lemma}
\label{rdominant}
    Let $W$ denote an irreducible $T$-submodule of the nucleus $\mathcal{N}$ with endpoint $r$. Let $\Omega$ denote an irreducible $\mathcal{H}$-module that is alpha-dominant with gap $r$. In view of Lemma \ref{HW}, the following are equivalent: 
    \begin{enumerate}[label=(\roman*)]
        \item $\widetilde{\Omega}=W$;
        \item $\Omega=W^{\mathcal{H}}$.
    \end{enumerate} 
\end{lemma}

\begin{proof}
(i)$\Rightarrow$(ii)
    By Lemma \ref{HW}, the $\mathcal{H}$-module $W^{\mathcal{H}}$ is irreducible. It suffices to show that $\Omega=W^{\mathcal{H}}$ is the only irreducible $\mathcal{H}$-module that is alpha-dominant with gap $r$ such that $\widetilde{\Omega}=W$. Let $\Omega'$ denote another such $\mathcal{H}$-module. By definition, the $\mathcal{H}$-module $\Omega'$ contains $W^{\mathcal{H}}$. Since both $\Omega'$ and $W^{\mathcal{H}}$ have type $(r,0,0)$, the $\mathcal{H}$-modules $\Omega'$ and $W^{\mathcal{H}}$ are isomorphic by Lemma \ref{unique}. Hence $\Omega'=W^{\mathcal{H}}$. The result follows.

    (ii)$\Rightarrow$(i) Clear.
\end{proof}

\begin{lemma}
\label{multr}
    The multiplicity of the irreducible $T$-submodules of $\mathcal{N}$ with endpoint $r$ is equal to 
    \begin{equation*}
        \mult_r=\binom{D}{r}_q-\binom{D}{r-1}_q.
    \end{equation*}
\end{lemma}

\begin{proof}
    By Lemma \ref{rdominant}, there is a one-to-one correspondence between (i) the irreducible $\mathcal{H}$-submodules of $\Psi$ that are alpha-dominant with gap $r$, and (ii) the irreducible $T$-submodules of $\mathcal{N}$ with endpoint $r$. The result follows from Lemma \ref{alphadominant}.
\end{proof}

\begin{theorem}
\label{Ndimension}
    The dimension of the nucleus $\mathcal{N}$ is equal to 
    \begin{equation*}
        \dim \mathcal{N}=\sum_{i=0}^{D}\binom{D}{i}_q.
    \end{equation*}
\end{theorem}
\begin{proof}
    Pick an integer $i$ such that $0\leq i\leq D$. If $0\leq i\leq \frac{D}{2}$, then by Proposition \ref{dimension} and Lemma \ref{multr},
    \begin{equation}
    \label{dimN_i}
        \dim \mathcal{N}_i=\sum_{r=0}^{i}\mult_r=\sum_{r=0}^{i}\Biggl(\binom{D}{r}_q-\binom{D}{r-1}_q\Biggr)=\binom{D}{i}_q.
    \end{equation}
    If $\frac{D}{2}< i\leq D$, then by Proposition \ref{dimension} and (\ref{dimN_i})
    \begin{equation}
    \label{dimN_D-i}
        \dim \mathcal{N}_i=\dim\mathcal{N}_{D-i}=\binom{D}{D-i}_q=\binom{D}{i}_q.
    \end{equation}

    Combine  (\ref{dimN_i}), (\ref{dimN_D-i}) and Lemma \ref{defN} to obtain the result.
\end{proof}

\section{The subgraph induced on $\Gamma_i(x)$ and the vectors $\alpha^{\vee}, \alpha^{\mathcal{N}}$}
Recall the Grassmann graph $\Gamma=J_q(N,D)$. 
Pick $x\in X$ and an integer $i$ ($0\leq i\leq D$). In this section we discuss the structure of the subgraph induced on $\Gamma_i(x)$, which will be useful for describing the nucleus $\mathcal{N}$. For each $\alpha\subseteq x$, we define vectors $\alpha^{\vee}, \alpha^{\mathcal{N}}$ in $V$; as we will see in Section \ref{somebases}, these vectors are related to $\mathcal{N}$.

\begin{lemma}
\label{y cap z}
    Pick adjacent vertices $y,z\in \Gamma_i(x)$. Then the subspace $y\cap z$ is either: (i) $/$-covered by each of $y$ and $z$, or (ii) $\backslash$-covered by each of $y$ and $z$.
\end{lemma}

\begin{proof}
    Note that $\Gamma_i(x)=P_{D-i,i}$. Hence, both $y$ and $z$ are contained in $P_{D-i,i}$. Also note that $y\cap z$ is covered by both $y$ and $z$. The result follows from Lemma \ref{coverlem}.
\end{proof}

We define the graph $\gamma_i(x)$ as follows: from the subgraph induced on $\Gamma_i(x)$, remove all the edges that satisfy type (i) in Lemma \ref{y cap z}. Our next goal is to describe the structure of $\gamma_i(x)$. We first present a useful lemma.

\begin{lemma}
\label{Dth}
    The subgraph of $\Gamma$ induced on $\Gamma_D(x)$ is connected.
\end{lemma}

\begin{proof}
Observe that $\Gamma_D(x)$ is regular with valency $a_D$, which is the largest eigenvalue of $\Gamma_D(x)$.  
In view of \cite[Proposition~3.1]{Biggs}, it suffices to show that $a_D$ has multiplicity $1$.  
Let $\mathcal{A}$ denote the adjacency matrix of the induced subgraph on $\Gamma_D(x)$.  
By construction, $\mathcal{A}$ is the restriction of $E_D^* A E_D^*$ to $\Gamma_D(x)$.  
Applying Lemma~\ref{EAE} with $r = D-d$ and $i = d$, we find that the eigenvalues of $\mathcal{A}$ are precisely the scalars $a_d(W)$, where $W$ runs over irreducible $T$-modules with endpoint $r$ and diameter $d$ such that $r+d=D$.  
In particular, $a_D$ arises from the primary $T$-module $W_0$.  
By Lemma~\ref{acomparison}, each $a_d(W)$ satisfies 
\begin{equation}\label{pf eq:Dth}
	a_d(W) \leq a_{r+d} = a_D.
\end{equation}
Using \eqref{kappa}, \eqref{sizebc} and setting $i=d$ in \eqref{aW}, we verify that the equality in \eqref{pf eq:Dth} holds if and only if $r=0$ and $d=D$.  
Thus only $W_0$ attains the value $a_D$ on $E_D^*W_0$.  
Since $W_0$ is thin, we have $\dim E_D^*W_0=1$.  
Therefore $a_D$ has multiplicity $1$. The result follows.
\end{proof}

Pick $\alpha\subseteq x$ such that $\dim \alpha=D-i$. Define 
\begin{equation}
\label{definegh}
    G_{\alpha}=\{y\in X\mid y\cap x=\alpha\}, \qquad \qquad H_{\alpha}=\{y\in X\mid \alpha\subseteq y\}. 
\end{equation}

\begin{remark}
\label{isomorphic}
    Observe that the subgraph induced on $H_{\alpha}$ is isomorphic to the graph $J_q(N-D+i,i)$; the isomorphism is obtained by taking the quotient of each subspace in $H_{\alpha}$ by $\alpha$. The set $G_{\alpha}$ is the $i$-th subconstituent of $H_{\alpha}$ with respect to $x$.
\end{remark}

\begin{lemma}
    The subgraph induced on $G_{\alpha}$ is connected.
\end{lemma}

\begin{proof}
    Combine Lemma \ref{Dth} and Remark \ref{isomorphic}.
\end{proof}
\begin{lemma}
\label{connected}
    The graph $\gamma_i(x)$ consists of $\binom{D}{i}_q$ connected components, where each component is a subgraph induced on $G_{\alpha}$ for $\alpha\subseteq x$, $\dim \alpha=D-i$.
\end{lemma}

\begin{proof}
    By Lemma \ref{qbinom}, there are $\binom{D}{D-i}_q=\binom{D}{i}_q$ subspaces of $x$ that have dimension $D-i$. In view of Lemma \ref{connected}, it suffices to show that for distinct $\alpha,\beta\subseteq x$ such that $\dim \alpha=\dim \beta=D-i$, the subgraphs induced on $G_{\alpha}$, $G_{\beta}$ form different connected components. Pick $y\in G_{\alpha}$, $z\in G_{\beta}$. Assume that there exists a path between $y$ and $z$ in $\gamma_i(x)$. Along the path there exist two adjacent vertices $y'\in G_{\alpha}$ and $z'\in G_{\beta}$. Observe that the edge $y'z'$ in $\Gamma_i(x)$ satisfies case (i) in Lemma \ref{y cap z}. Hence, the vertices $y',z'$ are not adjacent in $\gamma_i(x)$, which is a contradiction. The result follows.
\end{proof}

Pick $\alpha\subseteq x$. Next we define vectors $\alpha^{\vee},\alpha^{\mathcal{N}}\in V$. 
For $y \in X$, let $\widehat{y}$ denote the vector in $V$ with $1$ in the $y$-coordinate and $0$ elsewhere.
By the \emph{characteristic vector of $Y\subseteq X$}, we mean the vector
\begin{equation*}
    \sum_{y\in Y}\widehat{y}.
\end{equation*}

\begin{definition}
\label{vectordef}
    For $\alpha\subseteq x$, define the following vectors in $V$:
    \begin{equation*}
        \alpha^{\vee}=\sum_{y\in H_{\alpha}}\widehat{y}, \qquad \qquad \alpha^{\mathcal{N}}=\sum_{y\in G_{\alpha}}\widehat{y},
    \end{equation*}
    where $G_\alpha,H_{\alpha}$ are from (\ref{definegh}).
\end{definition}

\begin{lemma}
\label{vee to N}
    For $\alpha\subseteq x$ the following (\ref{veetoN}), (\ref{Ntovee}) hold: 
        \begin{equation}
            \label{veetoN}
            \alpha^{\vee}=\sum_{\alpha\subseteq \beta\subseteq x}\beta^{\mathcal{N}};
        \end{equation} 
        \begin{equation}
            \label{Ntovee}
            \alpha^{\mathcal{N}}=\sum_{\alpha\subseteq \beta\subseteq x}(-1)^{\dim\beta-\dim\alpha}q^{\binom{\dim\beta-\dim\alpha}{2}}\beta^{\vee}.
        \end{equation}
\end{lemma}

\begin{proof}
    Equation (\ref{veetoN}) follows from the definition. Next we prove (\ref{Ntovee}). For $\alpha\subseteq \beta\subseteq x$, write 
    \begin{equation*}
        C_{\alpha,\beta}=(-1)^{\dim \beta-\dim\alpha}q^{\binom{\dim \beta-\dim \alpha}{2}}.
    \end{equation*}
    Using (\ref{veetoN}) we obtain
    \begin{equation*}
        \sum_{\alpha\subseteq \beta\subseteq x}C_{\alpha,\beta}\beta^{\vee}=\sum_{\alpha\subseteq \beta\subseteq x}C_{\alpha,\beta}\sum_{\beta\subseteq \gamma\subseteq x} \gamma^{\mathcal{N}}=\sum_{\alpha\subseteq \gamma\subseteq x}\gamma^{\mathcal{N}}\sum_{\alpha\subseteq \beta\subseteq \gamma}C_{\alpha,\beta}.
    \end{equation*}

    We now show that 
    \begin{equation*}
        \sum_{\alpha\subseteq \beta\subseteq \gamma}C_{\alpha,\beta}=\begin{cases}
            1&\text{if $\gamma=\alpha$,}\\
            0&\text{if $\gamma\neq \alpha$.}
        \end{cases}
    \end{equation*}
    
    Observe that
    \begin{equation*}
        \sum_{\alpha\subseteq \beta\subseteq \gamma}C_{\alpha,\beta}=\sum_{j=0}^{\ell}(-1)^jq^{\binom{j}{2}}\binom{\ell}{j}_q,
    \end{equation*}
    where $\ell=\dim \gamma-\dim \alpha$. The result follows from Lemma \ref{identity1} in the appendix.
\end{proof}

\begin{lemma}
\label{alphaN,alphaV}
    For $\alpha\subseteq x$, 
    \begin{equation*}
        \alpha^{\mathcal{N}}= E^*_{D-\dim \alpha}\alpha^{\vee}.
    \end{equation*}
\end{lemma}

\begin{proof}
    Immediate from Definition \ref{vectordef}.
\end{proof}

\begin{lemma}
\label{independent}
    The following (i),(ii) hold:
    \begin{enumerate}[label=(\roman*)]
        \item $\{\alpha^{\mathcal{N}}\mid \alpha\subseteq x\}$ are linearly independent;
        \item $\{\alpha^{\vee}\mid \alpha\subseteq x\}$ are linearly independent.
    \end{enumerate}
\end{lemma}

\begin{proof}
    (i) The sets $\{G_\alpha \mid \alpha \subseteq x\}$ are mutually disjoint by Lemma~\ref{connected}, and $\alpha^{\mathcal{N}}$ is the characteristic vector of $G_\alpha$ by Definition~\ref{vectordef}. The result follows.
    
    (ii) Combine (i) and (\ref{veetoN}). 
\end{proof}

\section{The action of $A,A^*$ on the vectors $\alpha^\vee$, $\alpha^\mathcal{N}$}
Recall the Grassmann graph $\Gamma=J_q(N,D)$.
Pick $x \in X$ and write $T=T(x)$.
In this section we display the action of $A,A^*$ on the vectors $\alpha^\vee$, $\alpha^\mathcal{N}$.

\begin{theorem}
\label{Aaction}
    For $\alpha\subseteq x$,
    \begin{equation*}
        A\alpha^{\vee}=\theta_i \alpha^{\vee}+[D-i+1]\sum_{\text{$\alpha$ \rm{covers} $\beta$}}\beta^{\vee},
    \end{equation*}
    where $i=\dim\alpha$.
\end{theorem}

\begin{proof}
    We mimic the proof given in \cite[Theorem~9.1]{NT2024arxiv}. Recall the set $H_{\alpha}$ from (\ref{definegh}). Define
    \begin{equation*}
        S=\{y\in X\setminus H_{\alpha}\mid \text{$y$ is adjacent to at least one vertex in $H_{\alpha}$}\},
    \end{equation*}
    where $X\setminus H_{\alpha}$ is the complement of $H_{\alpha}$ in $X$. Note that the characteristic vector of $S$ is equal to
    \begin{equation*}
        \sum_{\text{$\alpha$ covers $\beta$}}(\beta^{\vee}-\alpha^{\vee}).
    \end{equation*}
    By combinatorial counting, each vertex in $H_{\alpha}$ is adjacent to exactly $q[D-i][N-D]$ vertices in $H_{\alpha}$. Also, each vertex in $S$ is adjacent to exactly $[D-i+1]$ vertices in $H_{\alpha}$. Hence,
    \begin{equation*}
        A\alpha^{\vee}=q[D-i][N-D]\alpha^{\vee}+[D-i+1]\sum_{\text{$\alpha$ covers $\beta$}}(\beta^{\vee}-\alpha^{\vee}).
    \end{equation*}

    Use Corollary \ref{cover} and (\ref{eigenvalues}) to obtain the result.
\end{proof}

\begin{theorem}
    For $\alpha\subseteq x$, 
    \begin{align*}
        A \alpha^{\mathcal{N}}=&[D-i]\bigl(q[N-D]-[D-i]\bigr)\alpha^{\mathcal{N}}+q^{2D-2i-1}[N-2D+i+1]\sum_{\substack{\beta\subseteq x,\\\beta \text{ \rm{covers} }\alpha}}\beta ^{\mathcal{N}}\\
        &\qquad \qquad \qquad \qquad \qquad \qquad +q^{D-i}\sum_{\substack{\beta\subseteq x,\\ \text{\rm{$\alpha$ covers $\alpha\cap \beta$,}}\\ \text{ \rm{$\beta$ covers $\alpha\cap \beta$}}}}\beta^{\mathcal{N}}+[D-i+1]\sum_{\alpha\text{ \rm{covers }}\beta}\beta^{\mathcal{N}},
    \end{align*}
    where $i=\dim \alpha$.
\end{theorem}

\begin{proof}
    We mimic the proof given in \cite[Theorem~9.2]{NT2024arxiv}. Recall the set $G_{\alpha}$ from (\ref{definegh}). Write
    \begin{equation*}
        A \alpha^{\mathcal{N}}=\sum_{z\in X}e_z \widehat{z},
    \end{equation*}
    where $e_z$ is the number of vertices in $G_{\alpha}$ that are adjacent to $z$. Pick $z\in X$ and define $\beta=z\cap x$. We now compute $e_z$ for $z\in X$ using combinatorial counting. In the table below, for each case given in the left column, we display the value of $e_z$ in the right column.
    \begin{center}
    \begin{tabular}{c|c}
        Case & $e_z$ \\
        \hline
        \\
        $\vert \dim \beta-\dim \alpha\vert\geq 2$ & $0$\\ \\
        $\beta=\alpha$ & $[D-i]\bigl(q[N-D]-[D-i]\bigr)$\\ \\
        $\beta$ covers $\alpha$ & $q^{2D-2i-1}[N-2D+i+1]$\\ \\
        $\alpha$ covers $\alpha\cap \beta$ and $\beta$ covers $\alpha\cap \beta$ & $q^{D-i}$\\ \\
        $\alpha$ covers $\beta$ & $[D-i+1]$
    \end{tabular}
    \end{center}
    The result follows.
\end{proof}

\begin{theorem}
\label{a*N}
    For $\alpha\subseteq x$, 
    \begin{equation*}
        A^* \alpha^{\mathcal{N}}=\theta^{*}_{D-\dim \alpha}\alpha^{\mathcal{N}}.
    \end{equation*}
\end{theorem}

\begin{proof}
    Immediate from the fact that $\alpha^{\mathcal{N}}\in E_{D-\dim \alpha}^{*}V$.
\end{proof}

\begin{theorem}
\label{A*vee}
    For $\alpha\subseteq x$,
    \begin{equation*}
        A^*\alpha^{\vee}=\theta^*_{D-i}\alpha^{\vee}+\frac{q^{-D+i+1}[N][N-1]}{[D][N-D]}\sum_{\gamma \text{ \rm{covers} } \alpha}\gamma^{\vee},
    \end{equation*}
    where $i=\dim \alpha$.
\end{theorem}

\begin{proof}
    Using Lemma \ref{vee to N}(i),
    \begin{equation}
        \label{A*eq1}
        A^*\alpha^{\vee}=A^*\sum_{\alpha\subseteq \beta \subseteq x}\beta^{\mathcal{N}}.
    \end{equation}
    By Theorem $\ref{a*N}$,
    \begin{equation}
    \label{A*eq2}
        A^*\beta^{\mathcal{N}}=\theta^*_{D-\dim \beta}\beta^{\mathcal{N}}.
    \end{equation}

    Combining (\ref{A*eq1}) and (\ref{A*eq2}) and Lemma \ref{vee to N}(ii), we obtain
    \begin{align*}
        A^*\alpha^{\vee}&=\sum_{\alpha\subseteq \beta\subseteq x}\theta^*_{D-\dim\beta} \sum_{\beta\subseteq \gamma\subseteq x}(-1)^{\dim \gamma-\dim \beta}q^{\binom{\dim \gamma-\dim \beta}{2}}\gamma^{\vee}\\
        &=\sum_{\alpha\subseteq \gamma\subseteq x}\gamma^{\vee}\sum_{\alpha\subseteq \beta \subseteq \gamma}\theta^{*}_{D-\dim \beta}(-1)^{\dim \gamma-\dim \beta}q^{\binom{\dim\gamma-\dim \beta}{2}}.
    \end{align*}

    Observe that
    \begin{align}
        \sum_{\alpha\subseteq \beta \subseteq \gamma}\theta^{*}_{D-\dim \beta}(-1)^{\dim \gamma-\dim \beta}q^{\binom{\dim\gamma-\dim \beta}{2}}&=\sum_{i=\dim \alpha}^{\dim \gamma}\theta^*_{D-i}(-1)^{\dim \gamma-i}q^{\binom{\dim \gamma-i}{2}}\binom{\dim \gamma-\dim \alpha}{\dim \gamma-i}_{q} \\
        &=\sum_{j=0}^{\dim \gamma-\dim \alpha}\theta^*_{D-\dim \gamma+j}(-1)^{j}q^{\binom{j}{2}}\binom{\dim \gamma-\dim \alpha}{j}_{q}.\label{sum}
    \end{align}
    Write $\ell=\dim \gamma-\dim \alpha$. Observe that if $\ell=0$, then the sum in (\ref{sum}) equals $\theta^*_{D-\dim \alpha}$. If $\ell=1$, then by (\ref{dual}), the sum in (\ref{sum}) equals
    \begin{equation*}
        \frac{q^{-D+\dim \alpha+1}[N][N-1]}{[D][N-D]}.
    \end{equation*}

    We now assume that $\ell\geq 2$. It suffices to show that the sum in (\ref{sum}) is equal to $0$. Using (\ref{dual}), we obtain
    \begin{equation}
    \label{split}
        \sum_{j=0}^{\ell}\theta^*_{D-\dim \gamma+j}(-1)^{j}q^{\binom{j}{2}}\binom{\ell}{j}_{q}=C_1\sum_{j=0}^{\ell}(-1)^{j}q^{\binom{j}{2}}\binom{\ell}{j}_{q}+C_2\sum_{j=0}^{\ell}(-1)^{j}q^{\binom{j}{2}-j}\binom{\ell}{j}_{q},
    \end{equation}
    where 
    \begin{equation*}
        C_1=-\frac{q[N-1]([D]+[N-D])}{(q-1)[D][N-D]},\qquad \qquad C_2=\frac{q^{-D+\dim \gamma+1}[N][N-1]}{(q-1)[D][N-D]}.  
    \end{equation*}
    By Lemmas \ref{identity1}, \ref{identity2} in the appendix, the right-hand side of (\ref{split}) is equal to $0$. The result follows. 
\end{proof}

\section{Some bases for the nucleus $\mathcal{N}$}
\label{somebases}
We continue our discussion of the Grassmann graph $\Gamma=J_q(N,D)$.
Pick $x \in X$ and write $T=T(x)$.
Recall the nucleus $\mathcal{N}$ of $\Gamma$.
In this section we display two bases for $\mathcal{N}$. Recall the vectors $\alpha^{\vee}$ and $\alpha^{\mathcal{N}}$ from Definition \ref{vectordef}.

\begin{lemma}
    For $\alpha\subseteq x$,
    \begin{equation}
    \label{e0}
        \alpha^{\vee}\in E_0V+E_1V+\cdots + E_{\dim \alpha}V.
    \end{equation}
\end{lemma}

\begin{proof}
    If $\alpha=x$, then (\ref{e0}) clearly holds. Now we assume that $\alpha\subsetneq x$. Write $i=\dim \alpha$. Note that $i<D$. By Theorem \ref{Aaction}, for $\gamma\subseteq x$ we have
    \begin{equation}
    \label{a-theta}
        (A-\theta_{\dim \gamma}I)\gamma^{\vee}\in \text{Span}\{\beta^{\vee}\mid \text{$\gamma$ covers $\beta$}\}.
    \end{equation}

    By (\ref{a-theta}) and induction on $i$, we obtain
    \begin{equation*}
    \label{aproduct}
        (A-\theta_0I)(A-\theta_1I)\cdots (A-\theta_{i}I)\alpha^{\vee}=0.
    \end{equation*}
    By (\ref{ea}), we have
    \begin{equation*}
        E_r\alpha^{\vee}=0\qquad \qquad (i<r\leq D).
    \end{equation*}
    The result follows.
\end{proof}

\begin{lemma}
    For $\alpha\subseteq x$, 
    \begin{equation}
    \label{e*0}
        \alpha^{\vee}\in E^*_0V+E^*_1V+\cdots + E^*_{D-\dim \alpha}V.
    \end{equation}
\end{lemma}

\begin{proof}
    By the definition of $\alpha^{\vee}$.
\end{proof}

\begin{lemma}
\label{ND-i}
    For $\alpha\subseteq x$, 
    \begin{equation*}
        \alpha^{\vee}\in \mathcal{N}_{D-\dim \alpha}.
    \end{equation*}
\end{lemma}

\begin{proof}
    Combine (\ref{e0}) and (\ref{e*0}) to get
    $$
    \alpha^{\vee}\in (E^*_0V+E^*_1V+\cdots + E^*_{D-\dim \alpha}V) \cap (E_0V+E_1V+\cdots + E_{\dim \alpha}V).
    $$
    The result follows from Definition \ref{Ni} and Proposition \ref{dimension}.
\end{proof}

\begin{corollary}
\label{acheck}
    For $\alpha\subseteq x$, 
    \begin{equation*}
        \alpha^{\vee}\in \mathcal{N}.
    \end{equation*}
\end{corollary}

\begin{proof}
    Immediate from Lemmas \ref{defN} and \ref{ND-i}.
\end{proof}

\begin{lemma}
\label{alphaN}
    For $\alpha\subseteq x$, 
    \begin{equation*}
        \alpha^{\mathcal{N}}\in E^*_{D-\dim \alpha}\mathcal{N}.
    \end{equation*}
\end{lemma}

\begin{proof}
    Immediate from Lemma \ref{alphaN,alphaV} and Corollary \ref{acheck}.
\end{proof}

\begin{corollary}
    For $\alpha\subseteq x$, 
    \begin{equation*}
        \alpha^{\mathcal{N}}\in \mathcal{N}.
    \end{equation*}
\end{corollary}

\begin{proof}
    Immediate from Lemma \ref{alphaN}.
\end{proof}

We now display two bases for $\mathcal{N}$.

\begin{theorem}
    Each of (i),(ii) below forms a basis for $\mathcal{N}$:
    \begin{enumerate}[label=(\roman*)]
        \item $\{\alpha^{\vee}\mid \alpha\subseteq x\}$;
        \item $\bigl\{\alpha^{\mathcal{N}}\mid \alpha\subseteq x\bigr\}$.
    \end{enumerate}
\end{theorem}

\begin{proof}
    (i) By Corollary \ref{acheck}, the set $\{\alpha^{\vee}\mid \alpha\subseteq x\}$ is contained in $\mathcal{N}$. By Lemma \ref{independent}(ii), these vectors are linearly independent. The number of vectors in the set is equal to the total number of subspaces of $x$; by Lemma \ref{qbinom}, this number is  \begin{equation*}
        \sum_{i=0}^{D}\binom{D}{i}_q.
    \end{equation*}
    The result follows from Theorem \ref{Ndimension}.

    (ii) By (i) and \eqref{veetoN}, the vectors $\{\alpha^{\mathcal{N}} \mid \alpha \subseteq x\}$ span $\mathcal{N}$. The result follows from Lemma \ref{independent}(i).
\end{proof}
\begin{remark}
    The transition matrices between the bases
    \begin{equation*}
        \{\alpha^{\vee}\mid \alpha\subseteq x\},\qquad \qquad \bigl\{\alpha^{\mathcal{N}}\mid \alpha\subseteq x\bigr\}
    \end{equation*}
    are given in (\ref{veetoN}), (\ref{Ntovee}).
\end{remark}

\section{Case $N=2D$}
\label{remark}
Recall the Grassmann graph $J_q(N,D)$.
In this section we make a remark on the case $N=2D$ and present some open problems.

Assume that $N=2D$. 
Pick $x \in X$ and write $T=T(x)$ and $\mathcal{H}=\mathcal{H}(x)$.
Let $\Omega$ denote an irreducible $\mathcal{H}$-module of type $(0,\beta,0)$, where $1\leq \beta\leq \frac{D}{2}$. 
Observe that $\Omega$ is not alpha-dominant, yet $\widetilde{\Omega}$ is an irreducible $T$-submodule of the nucleus $\mathcal{N}$. Hence Lemmas \ref{beta-alpha}--\ref{HW} do not hold in the case of $\Omega$. 
This observation suggests that a different method is required to determine the values of $\text{mult}_r$ ($0\leq r\leq \frac{D}{2}$) as defined in Definition \ref{multrdef}.

\begin{problem}
(i) For $0\leq r\leq \frac{D}{2}$, find the value of $\text{mult}_r$ when $N=2D$, and (ii) find the dimension of $\mathcal{N}$.
\end{problem}

\begin{problem}
    Describe some bases for the nucleus $\mathcal{N}$ when $N=2D$.
\end{problem}

\section{Appendix}
In this section we present some identities that involve $q$-binomial coefficients.

\begin{lemma}
    For nonnegative integers $\ell\geq j$ such that $\ell$ and $j$ are not both zero, we have
    \begin{equation}
    \label{identity}
        \binom{\ell}{j}_q=q^j\binom{\ell-1}{j}_q+\binom{\ell-1}{j-1}_q.
    \end{equation}
\end{lemma}

\begin{proof}
    Routine from (\ref{qbinomdef}).
\end{proof}

\begin{lemma}
    \label{identity1}
    For a nonnegative integer $\ell$, we have
    \begin{equation}
    \label{=0}
        \sum_{j=0}^{\ell}(-1)^jq^{\binom{j}{2}}\binom{\ell}{j}_q=\begin{cases}
            1&\text{if $\ell=0$,}\\
            0&\text{if $\ell>0$.}
            \end{cases}
    \end{equation}
\end{lemma}

\begin{proof} If $\ell=0$, then it is clear that (\ref{=0}) holds. Now assume that $\ell>0$. Then by (\ref{identity}),
    \begin{align*}
        \sum_{j=0}^{\ell}(-1)^jq^{\binom{j}{2}}\binom{\ell}{j}_q&=\sum_{j=0}^{\ell}(-1)^jq^{\binom{j}{2}}\Biggl(q^j\binom{\ell-1}{j}_q+\binom{\ell-1}{j-1}_q\Biggr)\\
        &=\sum_{j=0}^{\ell}(-1)^{j}q^{\binom{j+1}{2}}\binom{\ell-1}{j}_q+\sum_{j=0}^{\ell}(-1)^{j}q^{\binom{j}{2}}\binom{\ell-1}{j-1}_q\\
        &=\sum_{j=0}^{\ell-1}(-1)^{j}q^{\binom{j+1}{2}}\binom{\ell-1}{j}_q+\sum_{j=1}^{\ell}(-1)^{j}q^{\binom{j}{2}}\binom{\ell-1}{j-1}_q\\
        &=\sum_{j=0}^{\ell-1}(-1)^{j}q^{\binom{j+1}{2}}\binom{\ell-1}{j}_q+\sum_{j=0}^{\ell-1}(-1)^{j+1}q^{\binom{j+1}{2}}\binom{\ell-1}{j}_q\\
        &=0.
    \end{align*}
\end{proof}

\begin{lemma}
\label{identity2}
    For a positive integer $\ell\geq 2$, we have
    \begin{equation*}
        \sum_{j=0}^{\ell}(-1)^{j}q^{\binom{j}{2}-j}\binom{\ell}{j}_{q}=0.
    \end{equation*}
\end{lemma}

\begin{proof}
By (\ref{identity}), we have
    \begin{align*}
        \sum_{j=0}^{\ell}(-1)^{j}q^{\binom{j}{2}-j}\binom{\ell}{j}_{q}&=\sum_{j=0}^{\ell}(-1)^{j}q^{\binom{j}{2}-j}\Biggl(q^j \binom{\ell-1}{j}_q+\binom{\ell-1}{j-1}_q\Biggr)\\
        &=\sum_{j=0}^{\ell}(-1)^{j}q^{\binom{j}{2}}\binom{\ell-1}{j}_q+\sum_{j=0}^{\ell}(-1)^{j}q^{\binom{j}{2}-j}\binom{\ell-1}{j-1}_q\\
        &=\sum_{j=0}^{\ell-1}(-1)^{j}q^{\binom{j}{2}}\binom{\ell-1}{j}_q+\sum_{j=1}^{\ell}(-1)^{j}q^{\binom{j}{2}-j}\binom{\ell-1}{j-1}_q\\
        &=\sum_{j=0}^{\ell-1}(-1)^{j}q^{\binom{j}{2}}\binom{\ell-1}{j}_q+\sum_{j=0}^{\ell-1}(-1)^{j+1}q^{\binom{j+1}{2}-j-1}\binom{\ell-1}{j}_q\\
        &=\bigl(1-q^{-1}\bigr)\sum_{j=0}^{\ell-1}(-1)^{j}q^{\binom{j}{2}}\binom{\ell-1}{j}_q.
    \end{align*}
    The result follows from Lemma \ref{identity1}.
\end{proof}

\section*{Acknowledgements}
The authors would like to thank Professor Paul Terwilliger for suggesting this project and giving many valuable ideas throughout the preparation of this manuscript. 

J. Park was supported by the National Research Foundation of Korea (NRF) grant funded by the Korea government (MSIT) (RS-2024-00356153). The authors would like to thank the NRF and the MSIT for funding all the research travels that were relevant to this project.

\section*{Declarations}
\subsection*{Data Availability Statement}
No datasets were generated or analyzed during the current study.

\subsection*{Conflict of interest}
The author has no relevant financial or non-financial interests to disclose.

\newpage
\begin{multicols}{2} 
Jae-Ho Lee\\
Department of Mathematics and Statistics\\
University of North Florida\\
Jacksonville, FL 32224 USA\\
email: jaeho.lee@unf.edu
\columnbreak 

Jongyook Park\\
Department of Mathematics\\
Kyungpook National University\\
Daegu, 41566, Republic of Korea\\
email: jongyook@knu.ac.kr
\end{multicols}

Ian Seong\\
Department of Mathematics\\
Williams College \\
18 Hoxsey St \\
Williamstown, MA 01267-2680 USA \\
email: is11@williams.edu\\
\end{document}